\documentclass[12pt]{article}
\usepackage{amsthm}
\usepackage{amsmath}
\usepackage{amsfonts}
\usepackage{graphicx}
\usepackage{latexsym}
\usepackage{amssymb}
\usepackage{stmaryrd}
\usepackage{amscd}
\usepackage{booktabs}
\usepackage[small]{caption}

\newcommand{\deff}{\mbox{$\stackrel{\rm def}{=}$}}

\newcommand{\sbinom}[2]{\left[ \begin{array}{c} #1 \\ #2 \end{array} \right] }

\newcommand{\field}[1]{\mathbb{#1}}

\newcommand{\Z}{\field{Z}}

\newcommand{\R}{\field{R}}
\newcommand{\T}{\field{T}}

\newcommand{\cB}{{\cal B}}
\newcommand{\cC}{{\cal C}}
\newcommand{\cD}{{\cal D}}

\newcommand{\cF}{{\cal F}}

\newcommand{\cS}{{\cal S}}
\newcommand{\cT}{{\cal T}}
\newcommand{\cP}{{\cal P}}

\newcommand{\sP}{\field{P}}

\newcommand{\sG}{\field{G}}

\DeclareMathAlphabet{\mathbfsl}{OT1}{cmr}{bx}{it}
\newcommand{\uuu}{\kern-1pt\mathbfsl{u}\kern-0.5pt}
\newcommand{\vvv}{\kern-1pt\mathbfsl{v}\kern-0.5pt}

\newcommand{\myboxplus}{\kern1pt\mbox{\small$\boxplus$}}

\makeatletter \DeclareRobustCommand{\sbinom}{\genfrac[]\z@{}}
\makeatother
\newcommand{\G}[2]{\sbinom{{#1}\kern-1pt}{{#2}\kern-1pt}}
\newcommand{\Gq}[2]{\sbinom{{#1}\kern-0.25pt}{{#2}\kern-0.25pt}}

\newcommand{\Ps}{\smash{{\sP\kern-2.0pt}_q\kern-0.5pt(n)}}
\newcommand{\sPs}{\smash{{\sP\kern-1.5pt}_q(n)}}
\newcommand{\Ptwo}{\smash{{\sP\kern-2.0pt}_2\kern-0.5pt(n)}}
\newcommand{\Ptwom}{\smash{{\sP\kern-2.0pt}_2\kern-0.5pt(m)}}
\newcommand{\Ptwonm}{\smash{{\sP\kern-2.0pt}_2\kern-0.5pt(n+m)}}
\newcommand{\Ptwoa}{\smash{{\sP\kern-2.0pt}_2\kern-0.5pt(1)}}
\newcommand{\Ptwob}{\smash{{\sP\kern-2.0pt}_2\kern-0.5pt(2)}}
\newcommand{\Ptwoc}{\smash{{\sP\kern-2.0pt}_2\kern-0.5pt(3)}}
\newcommand{\Ptwod}{\smash{{\sP\kern-2.0pt}_2\kern-0.5pt(4)}}
\newcommand{\Ptwoe}{\smash{{\sP\kern-2.0pt}_2\kern-0.5pt(5)}}
\newcommand{\Ptwof}{\smash{{\sP\kern-2.0pt}_2\kern-0.5pt(6)}}
\newcommand{\Ptwokm}{\smash{{\sP\kern-2.0pt}_2\kern-0.5pt(2k-1)}}
\newcommand{\Pone}{\smash{{\sP\kern-2.5pt}_2\kern-0.5pt(n{-}1)}}

\newcommand{\Gr}{\smash{{\sG\kern-1.5pt}_q\kern-0.5pt(n,k)}}
\newcommand{\Gi}{\smash{{\sG\kern-1.5pt}_q\kern-0.5pt(n,i)}}
\newcommand{\Gj}{\smash{{\sG\kern-1.5pt}_q\kern-0.5pt(n,j)}}
\newcommand{\Grmk}{\smash{{\sG\kern-1.5pt}_q\kern-0.5pt(n,n-k)}}
\newcommand{\Grdk}{\smash{{\sG\kern-1.5pt}_q\kern-0.5pt(2k,k)}}
\newcommand{\Grekappa}{\smash{{\sG\kern-1.5pt}_q\kern-0.5pt(n,e+1-\kappa)}}
\newcommand{\Grtwoekappa}{\smash{{\sG\kern-1.5pt}_q\kern-0.5pt(n,2e+1-\kappa)}}
\newcommand{\Gremkappa}{\smash{{\sG\kern-1.5pt}_q\kern-0.5pt(n,e-\kappa)}}
\newcommand{\Gn}{\smash{{\sG\kern-1.5pt}_2\kern-0.5pt(n,n{-}1)}}
\newcommand{\Gnq}{\smash{{\sG\kern-1.5pt}_q\kern-0.5pt(n,n{-}1)}}
\newcommand{\Gone}{\smash{{\sG\kern-1.5pt}_2\kern-0.5pt(n,1)}}
\newcommand{\Gqone}{\smash{{\sG\kern-1.5pt}_q\kern-0.5pt(n,1)}}
\newcommand{\GTwo}{\smash{{\sG\kern-1.5pt}_2\kern-0.5pt(n,k)}}
\newcommand{\GTwonk}[2]{{\smash{{\sG\kern-1.5pt}_2\kern-0.5pt({#1},{#2})}}}
\newcommand{\Gnk}{\smash{{\sG\kern-1.5pt}_2\kern-0.5pt(n,n{-}k)}}
\newcommand{\Greone}{\smash{{\sG\kern-1.5pt}_q\kern-0.5pt(n,e{+}1)}}
\newcommand{\Gretwo}{\smash{{\sG\kern-1.5pt}_q\kern-0.5pt(n,e{+}2)}}

\newcommand{\be}[1]{\begin{equation}\label{#1}}
\newcommand{\ee}{\end{equation}}

\newcommand{\Cref}[1]{Co\-rol\-la\-ry\,\ref{#1}}

\newtheorem{theorem}{Theorem}[section]

\newtheorem{lemma}[theorem]{Lemma}

\newtheorem{cor}[theorem]{Corollary}

\theoremstyle{definition}  % The following theorem-like environments are in
\newtheorem{remark}{Remark}
\newtheorem{example}{Example}
\newtheorem{construction}{Construction}

\begin{document}

\bibliographystyle{plain}

\title{
\begin{center}
Tilings by $(0.5,n)$-Crosses and Perfect Codes
\end{center}
}
\author{
{\sc Sarit Buzaglo}\thanks{Department of Computer Science,
Technion, Haifa 32000, Israel, e-mail: {\tt
sarahb@cs.technion.ac.il}.} \and {\sc Tuvi
Etzion}\thanks{Department of Computer Science, Technion, Haifa
32000, Israel, e-mail: {\tt etzion@cs.technion.ac.il}.}}

\maketitle

\begin{abstract}
\hspace{2pt} The existence question for tiling of the
$n$-dimensional Euclidian space by crosses is well known. A few
existence and nonexistence results are known in the literature. Of
special interest are tilings of the Euclidian space by crosses
with arms of length one, known also as Lee spheres with radius
one. Such a tiling forms a perfect code. In this paper crosses
with arms of length half are considered. These crosses are scaled
by two to form a discrete shape. A tiling with this shape is also known
as a perfect dominating set. We prove that an integer tiling
for such a shape exists if and only if $n=2^t-1$ or $n=3^t-1$, where
$t>0$. A strong connection of these tilings to binary and ternary
perfect codes in the Hamming scheme is shown.
\end{abstract}

\vspace{0.5cm}

\noindent {\bf Keywords:} cross, perfect code, perfect dominating set, semicross, tiling.

\vspace{0.5cm}

\noindent {\bf AMS subject classifications:} 05B07, 11H31, 11H71, 52C22.

\footnotetext[1] { This research was supported in part by the
Israel Science Foundation (ISF), Jerusalem, Israel, under Grant
230/08. }

%%%%%%%%%%%%%%%%%%%%%%%%%%%%%%%%%%%%%%%%%%%%%%%%%%%%%%%%%%%%%%%%%%%%%%
%%%%%%%%%%%%%%%%%%%%%%%%%%%%%%%%%%%%%%%%%%%%%%%%%%%%%%%%%%%%%%%%%%%%%%
%%%%%%%%%%%%%%%%%%%%%%%%%%%%%%%%%%%%%%%%%%%%%%%%%%%%%%%%%%%%%%%%%%%%%%
\newpage
\section{Introduction}
\vspace{.5ex} \label{sec:introduction}

Packing and covering are two fundamental concepts in
combinatorics. Tiling is a concept which combines both packing and
covering and hence it attracts a substantial interest. Tiling of
the Euclidian space with specific shapes is one of the main
interest in this respect. Two of the shapes in this context are
the semicross and the cross. A $(k,n)$-\emph{semicross} is an
$n$-dimensional shape whose center is an $n$-dimensional unit cube
from which $n$ \emph{arms} consisting of $k$ $n$-dimensional unit
cubes are spanned in the $n$ positive directions. A
$(k,n)$-\emph{cross} is an $n$-dimensional shape whose center is
an $n$-dimensional unit cube from which $2n$ \emph{arms}
consisting of $k$ $n$-dimensional unit cubes are spanned in the
$n$ directions (one for the positive and one for the negative).
Examples of a $(2,3)$-cross and a $(2,3)$-semicross are given
in Figure~\ref{fig:CrossSemi}. Packing and tiling with semicrosses and crosses is a well
studied topic (see~\cite{Stein86,SteinSzabo} and
references therein).

\vspace{1.5cm}

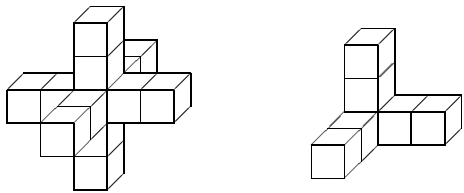
\begin{figure}[htbp]

%\centering

\setlength{\unitlength}{0.4mm}
\begin{picture}(0,0)(-200,10)
\resizebox{!}{20mm}{
\linethickness{.4 pt}
%%center
\put(50,0){\line(1,0){20}} \put(70,0){\line(0,1){20}}
\put(50,20){\line(1,0){20}} \put(50,0){\line(0,1){20}}
\put(50,20){\line(1,1){20}} \put(70,20){\line(1,1){20}}
\put(70,0){\line(1,1){20}} \put(60,30){\line(1,0){20}}
\put(80,30){\line(0,-1){20}} \put(90,40){\line(0,-1){20}}
\put(70,40){\line(1,0){60}}
%\put(70,40){\line(-1,0){40}}
%\put(30,40){\line(0,-1){20}}
%\put(30,40){\line(1,1){10}}
%\put(30,20){\line(1,0){20}}
\put(90,40){\line(0,1){40}} \put(70,40){\line(0,1){40}}
%\put(40,50){\line(1,0){30}}
%\put(50,40){\line(1,1){10}}
\put(90,40){\line(1,1){10}} \put(100,50){\line(0,1){40}}
\put(90,80){\line(1,1){10}} \put(90,80){\line(-1,0){20}}
\put(70,80){\line(1,1){10}} \put(80,90){\line(1,0){20}}
%\put(120,70){\line(-1,0){20}}
\put(100,50){\line(1,0){40}} \put(90,20){\line(1,0){40}}
%\put(90,20){\line(0,-1){40}}
%\put(90,-20){\line(-1,0){20}}
%\put(70,-20){\line(0,1){20}}
%\put(70,0){\line(1,0){20}}
%\put(90,-20){\line(1,1){10}}
%\put(100,-10){\line(0,1){30}}
%\put(90,0){\line(1,1){10}}
\put(130,20){\line(0,1){20}} \put(130,20){\line(1,1){10}}
\put(140,30){\line(0,1){20}} \put(110,20){\line(0,1){20}}
\put(110,40){\line(1,1){10}} \put(130,40){\line(1,1){10}}
%\put(120,70){\line(0,-1){20}}
%\put(110,60){\line(-1,0){10}}
%\put(110,60){\line(0,-1){10}}
\put(70,60){\line(1,0){20}} \put(90,60){\line(1,1){10}}
}
\end{picture}

\setlength{\unitlength}{.4mm}
\begin{picture}(0,0)(-110,-10)
\resizebox{!}{20mm}{
\linethickness{.4 pt}
%%center
\put(50,0){\line(1,0){20}} \put(70,0){\line(0,1){20}}
\put(50,20){\line(1,0){20}} \put(50,0){\line(0,1){40}}
\put(50,20){\line(1,1){20}}
\put(70,20){\line(1,1){20}}
\put(70,0){\line(1,1){20}}
\put(60,30){\line(1,0){20}}
\put(80,30){\line(0,-1){20}}
\put(90,40){\line(0,-1){20}}
\put(70,40){\line(1,0){60}}
\put(70,40){\line(-1,0){40}}
\put(30,40){\line(0,-1){20}}
\put(30,40){\line(1,1){10}}
\put(30,20){\line(1,0){20}}
\put(90,40){\line(0,1){40}}
\put(70,40){\line(0,1){40}}
\put(40,50){\line(1,0){30}}
\put(50,40){\line(1,1){10}}
\put(90,40){\line(1,1){30}}
\put(100,50){\line(0,1){40}}
\put(90,80){\line(1,1){10}}
\put(90,80){\line(-1,0){20}}
\put(70,80){\line(1,1){10}}
\put(80,90){\line(1,0){20}}
\put(120,70){\line(-1,0){20}}
\put(100,50){\line(1,0){40}}
\put(90,20){\line(1,0){40}}
\put(90,20){\line(0,-1){40}}
\put(90,-20){\line(-1,0){20}}
\put(70,-20){\line(0,1){20}}
\put(70,0){\line(1,0){20}}
\put(90,-20){\line(1,1){10}}
\put(100,-10){\line(0,1){30}}
\put(90,0){\line(1,1){10}}
\put(130,20){\line(0,1){20}}
\put(130,20){\line(1,1){10}}
\put(140,30){\line(0,1){20}}
\put(110,20){\line(0,1){20}}
\put(110,40){\line(1,1){10}}
\put(130,40){\line(1,1){10}}
\put(120,70){\line(0,-1){20}}
\put(110,60){\line(-1,0){10}}
\put(110,60){\line(0,-1){10}}
\put(70,60){\line(1,0){20}}
\put(90,60){\line(1,1){10}}
}
\end{picture}

\caption{A $(2,3)$-cross and a $(2,3)$-semicross.}

\label{fig:CrossSemi}

\end{figure}

As mentioned in~\cite{SteinSzabo} the origins of the study of the
cross and semicross are in several independent
sources~\cite{GoWe,Kar66,Stein67,Ulr57}, some of which are pure
mathematics and some are connected to coding theory. Semicross and
cross are two types of ``error spheres" as explained
in~\cite{Gol69}. Golomb and Welch~\cite{GoWe} proved that the
$(1,n)$-cross tiles the $n$-dimensional Euclidian space for all $n
\geq 1$. Such a tiling is a perfect code in the Manhattan metric
and if the tiling is periodic then it is also a perfect code in
the Lee metric. Their work inspired future work (see \cite{Etz11}
and references therein) on perfect codes in the Lee (and Manhattan)
metric.

As said before, packing and tiling with semicrosses and crosses
are well studied
topics~\cite{Charlebois,EH,GoWe,HS74,HS84,Loomis,Stein67,Stein72,Stein84,Szabo81,Szabo83,Szabo85}.
The results in these research works include bounds on the size of
the arms, constructions for such packings and tilings, parameters
for which such tilings cannot exist, lattice and non-lattice
tilings, etc. Recently, the topic has gained a new interest since
the $(k,n)$-semicross is the error sphere of the \emph{asymmetric
error model} associated with flash memories~\cite{CSB,KBE}, the
most advanced type of storage currently used.
Schwartz~\cite{Sch11} investigated lattice tilings with
generalized crosses and semicrosses in the connection of
\emph{unbalanced limited magnitude error model} for multi level
flash memories.

Not much is known about tiling of crosses with arms which are not
of integer length. Moreover, most tilings considered in the
literature are integer lattice tilings. In this paper we study the
existence of tiling of the $n$-dimensional Euclidian space with a
$(0.5,n)$-cross. The $(0.5,n)$-cross consists of one complete
(non-fractional) unit cube and $2n$ halves unit cubes. Usually, it
is more convenient to handle tiling with complete unit cubes.
Hence, we scale the $(0.5,n)$-cross by two to obtain a new shape,
which will be denoted in the sequel by $\Upsilon_n$. The shape $\Upsilon_n$
consists of $2^n (n+1)$ complete unit cubes. For the shape $\Upsilon_n$ we
will discuss only integer tiling (also known as $\Z$-tiling) which
is a tiling in which the centers of the unit cubes are placed on
points of $\Z^n$. We prove that such a tiling exists if and only
if $n=2^t-1$ or $n=3^t-1$, where $t > 0$. The related tiling with a
$(0.5,n)$-cross (obtained after scaling by $0.5$) will be called a
$\begin{tiny} \frac{1}{2}
\end{tiny} \Z$-tiling. We present an analysis of the structure
obtained from such a tiling. The tiling which is considered for
the $(0.5,n)$-cross is usually not an integer tiling. Moreover, we
discuss general tilings and not just lattice tilings as done in
most literature.

Dejter~\cite{Dej11} has brought to our attention that a tiling with
$\Upsilon_n$ is a perfect dominating set in~$\Z^n$. This problem
was considered by several authors, e.g.~\cite{ADH,Wei94} and
references therein. The problem that
we consider in this paper is one of the main open problems on this topic.

The rest of this paper is organized as follows. In
Section~\ref{sec:basic} we present the basic concepts used
throughout this paper. We define what is a tiling, a lattice
tiling, an integer tiling, and a periodic tiling. We discuss how
to handle a tiling with a $(0.5,n)$-cross. We discuss three
distance measures which are used in our discussion, the well known
Hamming distance, the Manhattan distance which is used for codes
in $\Z^n$, and a new distance measure needed for the
$(0.5,n)$-crosses, the cross distance. We also discuss how a
tiling with a $(0.5,n)$-cross can be analyzed. In
Section~\ref{sec:nonexistence} we make an analysis of a tiling
with a $(0.5,n)$-cross and prove that such a $\begin{tiny}
\frac{1}{2} \end{tiny} \Z$-tiling can exist only if $n=2^t-1$ or
$n=3^t-1$, where $t > 0$. Some necessary conditions for the existence of
such a tiling are given. In Section~\ref{sec:perfect} we show how
we can construct a tiling with a $(0.5,n)$-cross from a binary
perfect single-error-correcting code of length $n=2^t-1$ and
vice-versa. Finally, we show how to construct a tiling with a
$(0.5,n)$-cross from a ternary perfect single-error-correcting
code of length $\frac{n}{2}=\frac{3^t-1}{2}$.

\vspace{1.25ex}
%------------------------------------------------------------------------%
%                                                                        %
%    3.2. Complements in Projective Space                                %
%                                                                        %
%------------------------------------------------------------------------%
\section{Basic Concepts}
\vspace{-.25ex} \label{sec:basic}

For two vectors $X=(x_1 , x_2 , \ldots ,
x_n), Y=(y_1,y_2 , \ldots ,y_n) \in \R^n$ and a scalar $\alpha \in
\R$, the \emph{vector addition} $X+Y$ is defined by $X+Y \deff (x_1 + y_1 ,
x_2 + y_2 , \ldots , x_n + y_n )$ and the \emph{scalar
multiplication} is defined by $\alpha X
\deff (\alpha x_1 , \alpha x_2 , \ldots , \alpha x_n)$.
For two sets $\cS_1 \subseteq \R^n$ and $\cS_2 \subseteq \R^n$
the \emph{set addition} $\cS_1 + \cS_2$ is defined by
$\cS_1 + \cS_2 \deff \{ X+Y ~:~ X
\in \cS_1 ,~ Y \in \cS_2 \}$.
For a set $\cS \in \R^n$ and a vector $U \in \R^n$ the
\emph{translation} of $\cS$ by $U$ is $U + \cS \deff \{ U + X ~:~ X \in
\cS \}$. The \emph{multiplication} of $\cS$ by a scalar $\alpha
\in \R$ is defined by $\alpha \cS \deff \{ \alpha X ~:~ X \in \cS
\}$.

Let $\cS$ be an $n$-dimensional shape in the $n$-dimensional
Euclidian space ($\R^n$). We say that two translations of $\cS$, $\cS_1$
and $\cS_2$, are \emph{disjoint} if their intersection is
contained in at most an $(n-1)$-dimensional space. A \emph{tiling}
$\cT$ of the $n$-dimensional Euclidian space with the shape $\cS$
is a set of disjoint translations of $\cS$ such that each point
$(x_1,x_2 , \ldots , x_n ) \in \R^n$ is contained in at least one
translation of $\cS$.
%For a given shape $\cS$ we choose a fixed point
%which will be called the \emph{balanced point} of the shape. In
%any other translation of $\cS$ the balanced point will be chosen in the
%same relative position. The set of balanced points in the translations
%of $\cS$ contained in the tiling $\cT$ defines the tiling. Hence,
In the sequel a tiling $\cT$ will be defined by a set of points
$\T$ in $\R^n$ and a shape $\cS$. The point $X \in \T$ if and
only if the translation
$X+\cS \in \cT$. Henceforth, $\T$ will be called
a tiling if the shape $\cS$ is known. For example, the set of
points $\T = \{ (i,i+5j) ~:~ i,j \in \Z \}$ and the
$(2,2)$-semicross defines a tiling of $\R^2$ (part of the tiling
is presented in Figure~\ref{fig:semicrossTile}). A tiling $\T$
with a shape $\cS$ is called an \emph{integer tiling} (also a
$\Z$-\emph{tiling}) if $\T \subseteq \Z^n$. An $n$-dimensional
\emph{unit cube} centered at a point $(c_1 , c_2 , \ldots , c_n )
\in \R^n$ consists of the points $\{ (x_1, x_2 , \ldots , x_n )
~:~ | x_i - c_i | \leq 0.5, ~ 1 \leq i \leq n \}$. The
$n$-dimensional shape $\cS$ is a \emph{discrete shape} if $\cS$ is
a union of $n$-dimensional unit cubes, whose centers are in
$\Z^n$. Hence, a discrete $n$-dimensional shape $\cS$ can be
defined by a set of points from $\Z^n$. Therefore, in an integer
tiling with a discrete shape $\cS$, each point of $\Z^n$ is
contained in exactly one translation of $\cS$.

\begin{figure}[htbp]

%\centering

\setlength{\unitlength}{0.3mm}
\begin{picture}(0,0)(-240,0)
\resizebox{!}{0.2mm} { \linethickness{.4 pt}

%%center
\put(-20,0){\line(1,0){160}} \put(-20,0){\line(0,-1){160}}
\put(0,0){\line(0,-1){160}} \put(20,0){\line(0,-1){160}}
\put(40,0){\line(0,-1){160}} \put(60,0){\line(0,-1){160}}
\put(80,0){\line(0,-1){160}} \put(100,0){\line(0,-1){160}}
\put(120,0){\line(0,-1){160}} \put(140,0){\line(0,-1){160}}
\put(-20,-20){\line(1,0){160}} \put(-20,-40){\line(1,0){160}}
\put(-20,-60){\line(1,0){160}} \put(-20,-80){\line(1,0){160}}
\put(-20,-100){\line(1,0){160}} \put(-20,-120){\line(1,0){160}}
\put(-20,-140){\line(1,0){160}} \put(-20,-160){\line(1,0){160}}

\linethickness{2.0 pt}

\put(-20,-160){\line(1,0){60}} \put(-20,-160){\line(0,1){60}}
\put(-20,-100){\line(1,0){20}} \put(40,-160){\line(0,1){20}}
\put(40,-140){\line(-1,0){40}} \put(0,-100){\line(0,-1){40}}
\put(40,-140){\line(1,0){20}} \put(60,-140){\line(0,1){20}}
\put(60,-120){\line(-1,0){40}} \put(20,-120){\line(0,1){40}}
\put(20,-80){\line(-1,0){20}} \put(0,-80){\line(0,-1){20}}

\put(60,-160){\line(0,1){20}} \put(60,-120){\line(1,0){20}}
\put(80,-120){\line(0,-1){40}} \put(80,-120){\line(0,1){20}}
\put(80,-100){\line(-1,0){40}} \put(40,-100){\line(0,1){40}}
\put(40,-60){\line(-1,0){20}} \put(20,-60){\line(0,-1){20}}

\put(80,-100){\line(1,0){20}} \put(100,-100){\line(0,-1){40}}
\put(100,-140){\line(1,0){40}} \put(140,-140){\line(0,-1){20}}
\put(140,-160){\line(-1,0){60}}

\put(-20,-80){\line(1,0){20}} \put(-20,-60){\line(1,0){40}}
\put(-20,-60){\line(0,1){60}} \put(-20,0){\line(1,0){20}}
\put(0,0){\line(0,-1){40}} \put(0,-40){\line(1,0){60}}
\put(40,-40){\line(0,-1){20}} \put(20,0){\line(0,-1){20}}
\put(20,-20){\line(1,0){40}} \put(60,-20){\line(0,-1){20}}
\put(60,-40){\line(-1,0){20}} \put(60,-40){\line(0,-1){40}}
\put(60,-80){\line(1,0){40}} \put(100,-80){\line(0,-1){20}}
\put(100,-80){\line(1,0){20}} \put(120,-60){\line(0,-1){60}}
\put(140,-60){\line(-1,0){60}} \put(80,-60){\line(0,1){60}}
\put(80,-20){\line(-1,0){20}} \put(100,0){\line(-1,0){60}}
\put(100,0){\line(0,-1){40}} \put(100,-40){\line(1,0){40}}
\put(140,-40){\line(0,-1){60}} \put(120,0){\line(0,-1){20}}
\put(120,-20){\line(1,0){20}} \put(120,-120){\line(1,0){20}}

%\put(50,-20){\line(1,0){500}}

}
\end{picture}
\vspace{3.0cm} \caption{Tiling of $\R^2$ with a (2,2)-semicross.}
\label{fig:semicrossTile}
\end{figure}
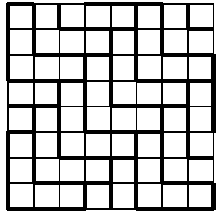

\begin{lemma}
\label{lem:shiftT} If $\T$ is a tiling with a shape $\cS$ and $U
\in \R^n$ then $U+\T$ is also a tiling with $\cS$.
\end{lemma}

By Lemma~\ref{lem:shiftT} we can assume that the origin, denoted
by $\mathbf{0}$, is always a point in the tiling. Therefore, given
a tiling $\T$ we assume without loss of generality that the origin is an element in
$\T$. For a set $\cS \subseteq \R^n$ and a permutation $\sigma$ of
$\{ 1,2,\ldots,n \}$, let $\sigma (\cS)
\deff \{ \sigma (X) ~:~ X \in \cS \}$.

\begin{lemma}
\label{lem:permute} If $\T$ is a tiling with an $n$-dimensional
shape $\cS$ and $\sigma$ is a permutation of $\{ 1,2,\ldots,n \}$
then $\sigma (\T)$ is a tiling with the $n$-dimensional shape
$\sigma (\cS)$.
\end{lemma}

The vector $(x_1 , x_2 , \ldots , x_n)$ is called the $r$-\emph{th
unit vector} and will be denoted by $e_r$ if $x_r =1$ and for all
$i \neq r$, $x_i =0$.
A set $\cS$ is called \emph{periodic} with period $p$ if $X \in
\cS$ implies that $X + \alpha \cdot p \cdot e_i \in \cS$, for all $\alpha \in
\Z$ and $1 \leq i \leq n$. A tiling $\T$ with the shape $\cS$ is a
periodic tiling if it is a periodic set. The following simple
lemma is left for the reader.

\begin{lemma}
A tiling $\T$ is a periodic with period $p$ if and only if $X \in
\T$ implies that $X + p \cdot e_i \in \T$ for all $i$, $1 \leq i \leq
n$.
\end{lemma}

A \emph{lattice} $\Lambda$ is a discrete, additive subgroup of the
real $n$-space $\R^n$,
\begin{equation*}
\Lambda \deff \{ u_1 v_1 + u_2v_2 + \cdots + u_n v_n ~:~ u_1,u_2,
\cdots,u_n \in \Z \}
\end{equation*}
where $\{ v_1,v_2,\ldots,v_n \}$ is a set of linearly independent
vectors in $\R^n$, i.e. the lattice has rank~$n$.
The set of vectors $\{ v_1,v_2,\ldots,v_n\}$ is
called \emph{the basis} for $\Lambda$, and the $n \times n$ matrix
$$
{\bf G} \deff \left[\begin{array}{cccc}
v_{11} & v_{12} & \ldots & v_{1n} \\
v_{21} & v_{22} & \ldots & v_{2n} \\
\vdots & \vdots & \ddots & \vdots\\
v_{n1} & v_{n2} & \ldots & v_{nn} \end{array}\right]
$$
having these vectors as its rows is said to be the \emph{generator
matrix} for $\Lambda$.

The {\it volume} of a lattice $\Lambda$, denoted by $V( \Lambda
)$, is inversely proportional to the number of lattice points per
a unit volume. More precisely, $V( \Lambda )$ may be defined as
the volume of the {\it fundamental parallelogram} $\Pi(\Lambda)$,
which is given by
$$
\Pi(\Lambda) \deff\ \{ \xi_1 v_1 + \xi_2 v_2 + \cdots + \xi_n
v_n~:~ 0 \leq \xi_i < 1,~ 1 \leq i \leq n \} ~.
$$
There is a simple expression for the volume of $\Lambda$, namely,
$V(\Lambda)=| \det {\bf G} |$.

A lattice $\Lambda$ is a \emph{lattice tiling} with $\cS$ if $\T
\deff \Lambda$ forms a tiling with $\cS$. A lattice tiling $\Lambda$ is an
\emph{integer lattice tiling} with $\cS$ if all entries of ${\bf
G}$ are integers. The following lemma is well known.

\begin{lemma}
If $\Lambda$ defines a lattice tiling with
the shape $\cS$ then $V(\Lambda)=|\cS|$, where
$|\cS|$ denote the volume of $\cS$.
\end{lemma}

A {\it code} $\cC$ of length $n$ over $\Z_q$ (respectively $\Z$) is a subset of
$\Z_q^n$ (respectively $\Z^n$). Let $\Lambda_n$ be the lattice generated by the
basis $\{ q \cdot e_i ~:~ 1 \leq i \leq n \}$. A code $\cC
\subseteq \Z_q^n$ can be viewed also as a subset of $\Z^n$. The
code $E(\cC)=\cC + \Lambda_n$ is the \emph{expanded code} of
$\cC$. If $E(\cC)$ is a tiling of $\Z^n$ with the shape $\cS$ then
we also call $\cC$ a tiling of $\Z_q^n$ with the shape $\cS$. A
tiling $\T \subseteq \Z^n$ with a period $p$ can be viewed as an
expanded code $E(\cC)$ of a code $\cC$ of length $n$ over
$\Z_p$, where $\cC = \T \cap \{ 0,1, \ldots , p-1 \}^n$. In
the sequel we denote
the set of integers in $\Z_p$ without the group structure by
$\tilde{\Z}_p \deff \{ 0,1, \ldots , p-1 \}$;
we will also refer to $\T$ as a code and to its elements as
codewords.

To handle a tiling with a $(0.5,n)$-cross we will need to use
three distance measures, the well known Hamming distance, the
Manhattan distance, and the new defined cross distance.

For every two given words $X,Y \in \Z_q^n$ the \emph{Hamming
distance} $d_H (X,Y)$ is the number of positions in which $X$ and
$Y$ differ, i.e.
$$
d_H (X,Y) \deff | \{ i ~:~ x_i \neq y_i, ~ 1 \leq i \leq n  \} |~.
$$

%For any two given points $X,Y \in \Z_q^n$ the \emph{Lee distance}
%$d_L (X,Y)$ is defined by
%$$
%d_L (X,Y) \deff \sum_{i=1}^n \min \{  x_i - y_i~(mod~q), y_i -
%x_i~(mod~q) \}~.
%$$

For every two given points $X,Y \in \Z^n$ the \emph{Manhattan
distance} $d_M (X,Y)$ is defined by
$$
d_M (X,Y) \deff \sum_{i=1}^n | x_i - y_i |~.
$$

For every two given points $X,Y \in \Z^n$ we defined the \emph{cross
distance} $d_C (X,Y)$ as follows
$$
d_C (X,Y) \deff \sum_{i=1}^n \max \{ 0, | y_i - x_i | -1  \}.
$$

\begin{remark}
The cross distance can be generalized for two points $X,Y \in
\R^n$. We will use this generalization only in this section.
\end{remark}

\begin{remark}
The Hamming distance is an association scheme, while the Manhattan distance is
only a metric distance and not an association scheme (see~\cite{McSl77} for the
definition of an association scheme). The cross distance is not a metric, but
it will be most important in the discussion of tilings with a
$(0.5,n)$-cross.
\end{remark}

For each distance measure we defined the \emph{weight} of a point
(word) $X$ to be the distance between $X$ and the point ${\bf 0}$.
The cross weight of a point $X$ will be denoted by $w_C (X)$.

A unit cube centered at $(c_1,c_2,\ldots,c_n) \in \R^n$ is a union
of two disjoint half unit cubes in one of the $n$ directions. For
the $r$-th direction one \emph{half unit cube} is defined by the
set of points $\{ (x_1, x_2 , \ldots , x_n ) ~:~ 0 \leq x_r -c_r
\leq 0.5,~ | x_i - c_i | \leq 0.5,~ 1 \leq i \leq n, ~ i \neq r
\}$ and a second \emph{half unit cube} is defined by the set of
points $\{ (x_1, x_2 , \ldots , x_n ) ~:~ -0.5 \leq x_r -c_r \leq
0,$ $ | x_i - c_i | \leq 0.5,~ 1 \leq i \leq n, ~ i \neq r \}$. A
$(0.5,n)$-cross is a unit cube to which two half unit cubes are
attached in the $r$-th direction for each $1 \leq r \leq n$, one
in its negative direction and one in its positive direction. It is
more convenient to handle shapes with complete unit cubes
(discrete shapes) and therefore we will scale the $(0.5,n)$-cross
by two to obtain a new shape which will be called $\Upsilon_n$. An
example of a $(0.5,3)$-cross and an $\Upsilon_3$ is given in
Figure~\ref{fig:HalfUps}. The complete unit cube in the
$(0.5,n)$-cross is transferred into an $n$-dimensional cube with
sides of length two in $\Upsilon_n$. This cube in $\Upsilon_n$
will be called the \emph{core} of $\Upsilon_n$; the core consists
of $2^n$ unit cubes. In the sequel we will be interested only in
integer tilings with $\Upsilon_n$. In such an integer tiling
$\Upsilon_n$ can be represented by $2^n (n+1)$ points of $\Z^n$
which are the centers of its $2^n (n+1)$ unit cubes.
%If $\C
%\subset \Z^n$ is the core of $\Upsilon_n$ then $\Upsilon_n = \{ U
%~:~ d_M (X,U)=1 ,~ X \in \C \}$. Given a point
%$(a_1,a_2,\ldots,a_n) \in \Z^n$ the set $\{ (c_1,c_2,\ldots,c_n)
%~:~ c_i \in \{ a_i -1 ,a_i \} ,~ 1 \leq i \leq n \}$ is an example
%for a core of $\Upsilon_n$; the set $\{ (c_1,c_2,\ldots,c_n) ~:~
%c_i \in \{ a_i +1 ,a_i \} ,~ 1 \leq i \leq n \}$ is another
%example for a core of $\Upsilon_n$.
Let $\T$ be a tiling with $\Upsilon_n$. We assume that if
$X=(x_1,x_2,\ldots,x_n) \in \T$  then the set
$\{ (c_1,c_2,\ldots,c_n) ~:~ c_i \in \{ x_i -1 ,x_i \},~ 1 \leq i
\leq n \}$ is the related core of the
translation $X+ \Upsilon_n$. The core of $\Upsilon_n$ is
$\{ -1,0 \}^n$ and
$\Upsilon_n \deff \{ U ~:~ d_M (X,U)=1 ,~ X \in \{ -1,0 \}^n \}$.
Even so we represent $\Upsilon_n$ as a set of points in $\Z^n$,
the integer tiling which we discussed is also
for the shape in the Euclidian space.
If $\T$ is a tiling with
$\Upsilon_n$ then $0.5 \T$ is a tiling with a $(0.5,n)$-cross.
Clearly, if for each $(x_1 , x_2 , \ldots , x_n) \in \T$, $x_i$ is
even for all $1 \leq i \leq n$, then also $0.5 \T$ is an integer
tiling. However, if there exists a point $(x_1 , x_2 , \ldots ,
x_n) \in \T$ such that for at least one $j$ we have that $x_j$ is
odd then $0.5 \T$ is not an integer tiling. To this end we define
a $\begin{tiny} \frac{1}{2}
\end{tiny} \Z$-tiling. A tiling $\T$ is a $\begin{tiny}
\frac{1}{2}
\end{tiny} \Z$-\emph{tiling} if $\T \subseteq 0.5 \Z^n$.

\vspace{2cm}

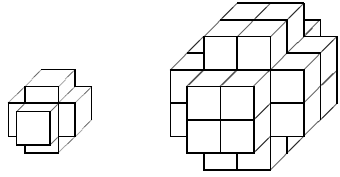
\begin{figure}[htbp]

\setlength{\unitlength}{.5mm}
\begin{picture}(-100,-100)(-115,0)
\resizebox{!}{10mm}{
\linethickness{.4 pt}
%%center
\put(50,0){\line(1,0){32}} \put(82,0){\line(0,1){32}}
\put(50,32){\line(1,0){32}} \put(50,0){\line(0,1){32}}
\put(50,32){\line(1,1){8}}
\put(58,40){\line(1,0){48}}
\put(82,32){\line(1,1){8}}
\put(82,0){\line(1,1){8}}
\put(90,8){\line(0,1){48}}
\put(90,8){\line(1,0){16}}
\put(106,8){\line(0,1){32}}
\put(106,8){\line(1,1){16}}
\put(106,40){\line(1,1){16}}
\put(90,40){\line(1,1){16}}
\put(90,56){\line(1,1){16}}
\put(58,56){\line(1,1){16}}
\put(58,56){\line(1,0){32}}
\put(58,56){\line(0,-1){16}}
\put(58,40){\line(-1,0){16}}
\put(42,40){\line(0,-1){32}}
\put(42,40){\line(1,1){16}}
\put(42,8){\line(1,0){8}}
\put(74,72){\line(1,0){32}}
\put(106,56){\line(1,0){16}}
\put(106,56){\line(0,1){16}}
\put(122,56){\line(0,-1){32}}
\put(90,8){\line(0,-1){16}}
\put(90,-8){\line(-1,0){32}}
\put(58,-8){\line(0,1){8}}
\put(90,-8){\line(1,1){16}}
}
\end{picture}

\setlength{\unitlength}{.4mm}
\begin{picture}(-100,-100)(-190,-10)
\resizebox{!}{20mm}{
\linethickness{.4 pt}
%%center
\put(50,0){\line(1,0){40}} \put(90,0){\line(0,1){40}}
\put(50,40){\line(1,0){40}} \put(50,0){\line(0,1){40}}
\put(50,40){\line(1,1){10}}
\put(90,40){\line(1,1){40}}
\put(90,0){\line(1,1){10}}
\put(100,10){\line(0,1){60}}
\put(100,70){\line(1,1){20}}
\put(60,70){\line(1,1){20}}
\put(80,90){\line(1,0){40}}
\put(60,70){\line(1,0){40}}
\put(60,50){\line(1,0){60}}
\put(60,50){\line(0,1){20}}
\put(120,70){\line(0,1){20}}
\put(60,50){\line(-1,0){20}}
\put(120,50){\line(0,-1){40}}
\put(120,50){\line(1,1){20}}
\put(120,70){\line(1,0){20}}
\put(100,10){\line(1,0){20}}
\put(120,10){\line(1,1){20}}
\put(140,30){\line(0,1){40}}
\put(100,10){\line(0,-1){20}}
\put(100,-10){\line(-1,0){40}}
\put(60,-10){\line(0,1){10}}
\put(100,-10){\line(1,1){20}}
\put(40,50){\line(0,-1){40}}
\put(40,10){\line(1,0){10}}
\put(40,50){\line(1,1){20}}
\put(130,80){\line(-1,0){10}}
\put(130,80){\line(0,-1){10}}

\put(70,0){\line(0,1){40}}
\put(50,20){\line(1,0){40}}
\put(90,20){\line(1,1){10}}
\put(70,40){\line(1,1){10}}
\put(100,30){\line(1,0){20}}
\put(80,50){\line(0,1){20}}
\put(80,70){\line(1,1){20}}
\put(120,30){\line(1,1){20}}
\put(70,80){\line(1,0){40}}
\put(130,20){\line(0,1){40}}
\put(130,60){\line(-1,0){20}}
\put(110,60){\line(0,1){20}}
\put(80,0){\line(0,-1){10}}
\put(50,30){\line(-1,0){10}}
\put(110,10){\line(0,-1){10}}
\put(60,60){\line(-1,0){10}}
}

\end{picture}

\caption{A $(0.5,3)$-cross and an $\Upsilon_3$.}

\label{fig:HalfUps}

\end{figure}

\begin{lemma}
The tiling $\T$ is an integer tiling with $\Upsilon_n$ if and only if $0.5
\T$ is a $\begin{tiny} \frac{1}{2} \end{tiny} \Z$-tiling with a
$(0.5,n)$-cross.
\end{lemma}

Given a set $\T \subset \Z^n$, we would like to know whether $\T$
is a tiling with $\Upsilon_n$. To show that~$\T$ is a tiling we
have to prove
\begin{itemize}
\item [$\boldsymbol{(\cP.1)}$] For each point $Y \in \Z^n$ there
exists a translation $\cS_1$ of $\Upsilon_n$ in the tiling such that
$\cS_1$ contains~$Y$.

\item [$\boldsymbol{(\cP.2)}$] A point $Y \in \Z^n$ is not
contained in more than one translation of $\Upsilon_n$ in the tiling,
i.e. for each two translations $\cS_1$ and $\cS_2$ of $\Upsilon_n$ in
the tiling we have $\cS_1 \cap \cS_2 = \varnothing$.
\end{itemize}

A set $\T \subset \Z^n$ is a \emph{covering} with $\Upsilon_n$ if
it satisfies property $\boldsymbol{(\cP.1)}$ and it is a
\emph{packing} with $\Upsilon_n$ if it satisfies property
$\boldsymbol{(\cP.2)}$. A tiling is clearly both a covering and a
packing.

The following two lemmas are immediate results from the definition
of $\Upsilon_n$.
\begin{lemma}
If $\cS$ is a translation of $\Upsilon_n$ and $X \in \cS$ is not a core
point of $\cS$ then there exists a core point $Y \in \cS$ such
that $d_M (X,Y)=1$.
\end{lemma}
\begin{lemma}
If $\cS_1$ and $\cS_2$ are two translations of $\Upsilon_n$ for which
$\cS_1 \cap \cS_2 \neq \varnothing$ then there exists a point $X
\in \cS_1 \cap \cS_2$ which is not in the core of $\cS_1$.
\end{lemma}
\begin{cor}
If $\cS_1$ and $\cS_2$ are two translations of $\Upsilon_n$ for which
$\cS_1 \cap \cS_2 \neq \varnothing$ then there exist two core
points $X_1 \in \cS_1$ and $X_2 \in \cS_2$ such that $d_M (X_1
,X_2) \leq 2$.
\end{cor}
\begin{lemma}
If $\cS_1$ and $\cS_2$ are two translations of $\Upsilon_n$ for which
there exist two core points $X_1 \in \cS_1$ and $X_2 \in \cS_2$
such that $d_M (X_1 ,X_2) \leq 2$, then $\cS_1 \cap \cS_2 \neq
\varnothing$.
\end{lemma}
\begin{proof}
If $d_M (X_1 ,X_2) \leq 2$ then there exists a point $Y \in \Z^n$
such that $d_M (X_1 , Y) \leq 1$ and $d_M(X_2 ,Y) \leq 1$. By
definition $Y \in \cS_1 \cap \cS_2$.
\end{proof}
\begin{cor}
\label{cor:dis_iff_Man} Let $\cS_1$ and $\cS_2$ be two translations of
$\Upsilon_n$. Then $\cS_1 \cap \cS_2 = \varnothing$ if and only if for
every two core points $X_1 \in \cS_1$ and $X_2 \in \cS_2$ we have
$d_M (X_1 ,X_2) \geq 3$.
\end{cor}

\begin{theorem}
\label{thm:P2} Let $\cS_1 =X+ \Upsilon_n$
and $\cS_2 =Y+\Upsilon_n$, where $X,Y \in \Z^n$, be two translations of
$\Upsilon_n$. Then $\cS_1 \cap
\cS_2 = \varnothing$ if and only if $d_C (X,Y) \geq 3$.
\end{theorem}
\begin{proof}
Let $\tilde{X} =(\tilde{x}_1,\tilde{x}_2,\ldots,\tilde{x}_n)$ and
$\tilde{Y}=(\tilde{y}_1,\tilde{y}_2,\ldots,\tilde{y}_n)$ be the
centers of mass of $\cS_1$ and $\cS_2$, respectively. Clearly,
$\tilde{X}$ and $\tilde{Y}$ are in $(0.5,0.5, \ldots , 0.5) +
\Z^n$. The core points of $\cS_1$ are $\{ (c_1 ,c_2 ,\ldots ,c_n)
~:~ c_i \in \{ \tilde{x}_i - 0.5 , \tilde{x}_i +0.5 \} \}$ and the
core points of $\cS_2$ are $\{ (c_1 ,c_2 ,\ldots ,c_n) ~:~ c_i \in
\{ \tilde{y}_i - 0.5 , \tilde{y}_i +0.5 \} \}$. Let
$X'=(x'_1,x'_2,\ldots,x'_n)$ and $Y'=(y'_1,y'_2,\ldots,y'_n)$ be
the two core points of $\cS_1$ and $\cS_2$, respectively, defined
as follows. If $\tilde{x}_i = \tilde{y}_i$ then $x'_i
\deff \tilde{x}_i +0.5$ and $y'_i \deff \tilde{y}_i +0.5$. If
$\tilde{x}_i < \tilde{y}_i$ then $x'_i \deff \tilde{x}_i +0.5$ and
$y'_i \deff \tilde{y}_i -0.5$. If $\tilde{x}_i > \tilde{y}_i$ then
$x'_i \deff \tilde{x}_i -0.5$ and $y'_i \deff \tilde{y}_i +0.5$.
Clearly, $d_C (X,Y) = d_C(\tilde{X},\tilde{Y})=d_M (X',Y')$ and
for any two core points $\hat{X} \in \cS_1$ and $\hat{Y} \in
\cS_2$ we have that $d_M (\hat{X} ,\hat{Y}) \geq d_M(X',Y')$. Now,
by Corollary~\ref{cor:dis_iff_Man} we have that $\cS_1 \cap \cS_2
= \varnothing$ if and only if $d_C (X,Y) \geq 3$.
\end{proof}

\begin{cor}
\label{cor:P2} The set $\T$ induces a packing of the
$n$-dimensional Euclidian space with $\Upsilon_n$ if and only if
for every two elements $X,Y \in \T$, we have $d_C (X,Y) \geq 3$.
\end{cor}

To prove that a set is a tiling with $\Upsilon_n$ we will have to
show that it satisfies properties $\boldsymbol{(\cP.1)}$ and
$\boldsymbol{(\cP.2)}$. For this purpose we will have to show that
each point of $\Z^n$ is contained (covered) in exactly one trnaslation
$\cS$ of $\Upsilon_n$ in the tiling. A point $A \in \Z^n$ is
\emph{covered} by a codeword $X$ in a tiling $\T$ if $A$ is
contained in the translation $X+\Upsilon_n$.
In this case we say that $X$ \emph{covers} $A$.

Given a tiling $\T$ with $\Upsilon_n$ it has to satisfy properties
$\boldsymbol{(\cP.1)}$ and $\boldsymbol{(\cP.2)}$. By considering
how each point $A \in \Z^n$ is covered by a codeword $X \in \T$
(property $\boldsymbol{(\cP.1)}$) we will discover the structure
of~$\T$. To this end we will also use property
$\boldsymbol{(\cP.2)}$, i.e. for each two codewords $X,Y \in \T$
we have that $d_C (X,Y) \geq 3$ (by
Corollary~\ref{cor:P2}).

\vspace{2ex}

%@@@@@@@@@@@@@@@@@@@@@@@@@@@@@@@@@@@@@@@@@@@@@@@@@@@@@@@@@@@@@@@@@@@@@@@@%
%                                                                        %
%   5. Nonlattice tilings                                      %
%                                                                        %
%@@@@@@@@@@@@@@@@@@@@@@@@@@@@@@@@@@@@@@@@@@@@@@@@@@@@@@@@@@@@@@@@@@@@@@@@%
\section{The Nonexistence of other Integer Tilings}
\vspace{-.25ex} \label{sec:nonexistence}

In this section we will prove that an integer tiling $\T$ with
$\Upsilon_n$ exists only if $n=2^t-1$ or $n=3^t-1$, for some $t > 0$. In
subsection~\ref{sec:Tiling_odd} we prove this claim for odd $n$ if $\T$
is an integer tiling and for all $n$ if $\T$ is
a lattice tiling. In subsection~\ref{sec:Tiling_even} we
complete the proof for even~$n$. We will obtain this goal by
proving that given a tiling $\T$ with $\Upsilon_n$, certain
elements of~$\Z^n$ must be contained in $\T$. It will be proved by
considering how elements with a small cross weight are covered. For
the rest of this section let $\T$ be a tiling with $\Upsilon_n$.
We remind that without loss of generality we assumed that ${\bf 0} \in \T$ and hence
by Corollary~\ref{cor:P2}, if $X,Y \in \T \setminus \{ {\bf 0}
\}$, where $X \neq Y$, then $w_C (X) \geq 3$, $w_C (Y) \geq 3$, and $d_C
(X,Y) \geq 3$.

\subsection{Tiling for odd $n$ and lattice tiling}
\label{sec:Tiling_odd}

In this subsection we first prove that
for every $r$, $1 \leq r \leq n$, the point $2 e_r$,
is covered by
either $4e_r$ or $3e_r +2e_s$ for some $s \neq r$.
Then we prove that if $3e_r +2e_s$ is a codeword then
$n$ is even, which will imply that if $n$ is odd
then $n=2^t-1$, for some
$t > 0$. Finally, we prove that two codewords of the
form $3e_r +2e_s$ are disjoint, i.e. have disjoint set of
nonzero coordinates. This will lead to the main result
only for a lattice tiling.
The first lemma is
an immediate result from the
definition of $\Upsilon_n$.
\begin{lemma}
\label{lem:cover_cond} Let $X \in \T$ and $A=(a_1,a_2,\ldots,a_n)
\in \Z^n$. The point $A$ is covered by $X$ if and only if
$x_i\in\{a_i-1,a_i,a_i+1,a_i+2\}$, for $1 \leq i \leq n$, and for at
most one $i$ we have $x_i\in \{a_i-1,a_i+2\}$.
\end{lemma}

Let $\cD_1$ be the set of points from $\{0,1,2,3\}^n$ in which 2
and 3 appear exactly once.

\begin{lemma}
\label{lem:codewords32} If  $X\in \cD_1 \cap \T$ then $X=3e_r
+2e_s$ for some $r \neq s$.
\end{lemma}
\begin{proof}
Assume without loss of generality that $X=(3,2,1,x_4 , \ldots , x_n)$, where $x_i
\in\{0,1\}$, for $4 \leq i \leq n$. The point $A=(1,1,-1,0, \ldots ,0)$ is
covered by a codeword $Y\in \T$. By Lemma~\ref{lem:cover_cond} we
have that $Y \not\in \{ X , \mathbf{0} \}$ and we can distinguish
between three cases:

\noindent {\bf Case 1}: If $y_i\in \{a_i,a_i+1\}$ for all $i$,
$1 \leq i \leq n$, then
$w_C(Y)\leq 2$, a contradiction.

\noindent {\bf Case 2}: There exists a $j$ such that $y_j=a_j-1$
and $y_i\in \{a_i,a_i+1\}$ for all $i\neq j$. Since $w_C (Y) \geq
3$ it follows that $j=3$ and hence $Y=(2,2,-2, y_4 , \ldots , y_n
)$, where $y_i \in \{ 0,1 \}$, for $4 \leq i \leq n$. This implies
that $d_C (X,Y)=2$, a contradiction.

\noindent {\bf Case 3}: There exists a $j$ such that $y_j=a_j+2$
and $y_i\in \{a_i,a_i+1\}$ for all $i\neq j$. Since $w_C (Y) \geq
3$ it follows that $j \neq 3$. Without loss of generality it implies that
$Y$ can take one of the following forms
\begin{itemize}
\item
$Y=(3,2,y_3 , y_4 , \ldots , y_n)$ or $Y=(2,3,y_3 , y_4 , \ldots ,
y_n)$, where $y_3 \in \{ -1 ,0 \}$ and $y_i \in \{ 0,1 \}$, for $4 \leq i
\leq n$.
\item
$Y=(2,2,y_3 ,2, y_5 , \ldots , y_n)$, where $y_3 \in \{ -1 ,0 \}$
and $y_i \in \{ 0,1 \}$, for $5 \leq i \leq n$.
\end{itemize}
Both forms implies that $d_C (X,Y) \leq 2$, a contradiction.

Therefore, there is no codeword $Y\in \T$ which covers $A$, a
contradiction. Thus, if ${X\in \cD_1 \cap \T}$ then $X=3e_r +2e_s$
for some $r \neq s$.
\end{proof}

Let $\cD_2$ be the set of points from $\{0,1,4\}^n$ in which 4
appears exactly once.

\begin{lemma}
\label{lem:codewords4} If  $X\in \cD_2 \cap \T$ then $X=4e_r$ for
some $1 \leq r \leq n$.
\end{lemma}
\begin{proof}
Assume without loss of generality that $X=(4,1,x_3 , \ldots , x_n)$, where $x_i
\in\{0,1\}$, for $3 \leq i \leq n$. The point $A=(1,1,0, \ldots ,0)$ is
covered by a codeword $Y\in \T$. By Lemma~\ref{lem:cover_cond} we
have that $Y \not\in \{ X , \mathbf{0} \}$ and we can distinguish
between two cases:

\noindent {\bf Case 1}: If $y_i\in \{a_i,a_i+1\}$ for all $i$, $1 \leq i \leq n$,
with a possible exception for at most one $j$, for which
$y_j=a_j-1$, then $w_C(Y)\leq 2$, a contradiction.

\noindent {\bf Case 2}: There exists a $j$ such that $y_j=a_j+2$
and $y_i\in \{a_i,a_i+1\}$ for all $i\neq j$. Without loss of generality
it implies that
$Y$ can take one of the following forms
\begin{itemize}
\item
$Y=(3,2,y_3 , \ldots , y_n)$, $Y=(2,3,y_3 , \ldots , y_n)$,
where $y_i \in \{ 0,1 \}$ for $3 \leq i \leq n$.

\item
$Y=(2,2,2 , y_4 , \ldots , y_n)$, where
$y_i \in \{ 0,1 \}$ for $4 \leq i \leq n$.
\end{itemize}
Hence, $d_C (X,Y) \leq 2$,
a contradiction.

Therefore, there is no codeword $Y\in \T$ which covers $A$, a
contradiction. Thus, if $X\in \cD_2 \cap \T$ then $X=4e_r$ for
some $1 \leq r \leq n$.
\end{proof}
\begin{cor}
\label{cor:cover2ei} For each $r$, $1 \leq r \leq n$, the point $2
e_r$ is covered by a codeword $X \in \T$, where either $X=4e_r$ or
$X=3e_r +2e_s$ for some $s \neq r$.
\end{cor}
\begin{proof}
By Lemma~\ref{lem:cover_cond}, $X$ is not the all-zero codeword.
Moreover, since $w_C(X) \geq 3$ it can be easily verified that
either $X \in \cD_1$ or $X \in \cD_2$. It follows from
Lemmas~\ref{lem:codewords32} and~\ref{lem:codewords4} that
either $X=4e_r$ or $X=3e_r +2e_s$ for some $s \neq r$.
\end{proof}

Let $\cD_3$ be the set of points from $\{0,1,2\}^n$ in which 2
appears exactly three times.

\begin{lemma}
\label{lem:codewords1222} If $X=3e_r +2e_s \in \T$ then for every
$k\not\in \{ r,s \}$ there exists a unique $j\not\in \{ r,s,k \}$
and a codeword $Y\in \cD_3 \cap \T$ such that
$y_r=1,y_s=y_k=y_j=2$.
\end{lemma}
\begin{proof}
Let $k\not\in \{ r,s \}$ and consider the point $A=e_r+e_s+e_k$.
Without loss of generality we assume that
$r=1$, $s=2$, and $k=3$, i.e. $X=(3,2,0,\ldots,0)$ and
$A=(1,1,1,0\ldots,0)$. The point
$A$ is covered by a codeword $Y\in \T$.
By Lemma~\ref{lem:cover_cond} we have that
$Y \not\in \{ X , \mathbf{0} \}$ and we can distinguish between
three cases:

\noindent {\bf Case 1}: If $y_i\in \{a_i,a_i+1\}$ for $1 \leq i \leq n$, then
since $w_C(Y) \geq 3$ it follows that $Y=(2,2,2,y_4 , \ldots ,
y_n)$, where $y_i \in \{ 0,1 \}$, for $4 \leq i \leq n$. Hence,
$d_C(X,Y) =1$, a contradiction.

\noindent {\bf Case 2}: There exists a $j$ such that $y_j=a_j-1$
and $y_i\in \{a_i,a_i+1\}$ for all $i\neq j$. If $j\leq 3$ then
$w_C(Y)\leq 2$, a contradiction. If $j > 3$ then since $w_C(Y)\geq
3$ it follows that $Y=(2,2,2,y_4 , \ldots , y_n)$, where
$y_i \in \{ -1,0,1 \}$
for $4 \leq i \leq n$, and hence $d_C(X,Y) = 1$, a
contradiction.

\noindent {\bf Case 3}: There exists a $j$ such that $y_j=a_j+2$
and $y_i\in \{a_i,a_i+1\}$ for all $i\neq j$. If $j \leq 3$ then
since $w_C(Y) \geq 3$ and $d_C(X,Y)\geq 3$ it follows that
$Y=(1,2,3,y_4 , \ldots , y_n)$, where $y_i \in \{ 0, 1\}$, for $4 \leq
i \leq n$, a contradiction to Lemma~\ref{lem:codewords32}.

\vspace{0.2cm}

Therefore, there exists a $j > 3$ such that $y_j=a_j+2$ and
$y_i\in \{a_i,a_i+1\}$ for all $i\neq j$. Without loss of generality we assume that
$j=4$. Since $w_C(Y) \geq 3$ and $d_C(X,Y)\geq 3$ it follows that
$Y=(1,2,2,2,y_5 , \ldots , y_n)$, where $y_i \in \{ 0,1 \}$, for $5
\leq i \leq n$. The uniqueness of $j$ follows from the fact that
if there exists another $j$ and a related codeword $Y'$ then
$d_C(Y,Y') \leq 2$.
\end{proof}

\begin{cor}
\label{cor:even32} If $3e_r +2e_s \in \T$ then $n$ is even.
\end{cor}
\begin{proof}
By Lemma~\ref{lem:codewords1222} all coordinates except for $r$ and $s$
should be paired, in disjoint pairs (such a pair
$\{ k,j \}$ induces a codeword of the form
$Y=(y_1,y_2, \ldots , y_n) \in \cD_3 \cap \T$,
where $y_r=1,y_s=y_k=y_j=2$). Thus, $n$ is even.
\end{proof}
From Corollaries~\ref{cor:cover2ei} and~\ref{cor:even32} we infer
that
\begin{cor}
\label{cor:odd4} If $n$ is odd then for all $X \in \T$ and $1 \leq
r \leq n$ we have $X+4e_r \in \T$, i.e. $\T$ is a periodic tiling
with period 4.
\end{cor}

\begin{theorem}
\label{thm:odd_tiling} If $\T$ is an integer tiling with
$\Upsilon_n$, where $n$ is an odd integer, then $n=2^t-1$ for some
$t > 0$.
\end{theorem}
\begin{proof}
By Corollary~\ref{cor:odd4} we have that $\T$ is a periodic tiling
with period 4. Therefore, the size of $\Upsilon_n$ divides $4^n$.
The size of $\Upsilon_n$ is $2^n (n+1)$ and hence $n=2^t-1$ for some $t > 0$.
\end{proof}

%\begin{cor}
%If $\T'$ is a $\begin{tiny} \frac{1}{2}
%\end{tiny} \Z$-tiling with $(0.5,n)$-cross, where $n$ is an
%odd integer, then $n=2^t-1$.
%\end{cor}

\begin{lemma}
\label{lem:disjoint32} If there exist two distinct codewords
$X=3e_i +2e_j$ and $X'=3e_r +2e_s$ in $\T$ then $\{ i,j \} \cap \{
r,s \} = \varnothing$.
\end{lemma}
\begin{proof}
Without loss of generality we assume that $i=1$ and $j=2$. Since $d_C(X,X') \geq 3$
it follows that $r \neq 1$ and $X' \neq 3e_2 +2e_1$. If $r=2$ or
$s=2$ then without loss of generality we assume that $X'=(0,3,2,0, \ldots ,0)$ or
$X'=(0,2,3,0, \ldots ,0)$. By Lemma~\ref{lem:codewords1222} we
have a codeword $Y=(1,2,2,y_4 , \ldots , y_n) \in \cD_3 \cap \T$.
It implies that $d_C (X',Y)=1$, a contradiction. The case where
$s=1$ and $r>2$ is symmetric to the case where $r=2$ and $s>2$.
\end{proof}
From Corollary~\ref{cor:cover2ei} and Lemma~\ref{lem:disjoint32}
we have that
\begin{cor}
\label{cor:one32and4} If $3e_r +2e_s \in \T$ then $4e_s \in \T$.
\end{cor}

\begin{theorem}
If $\T$ is an integer lattice tiling with $\Upsilon_n$ then either
$n=2^t-1$ or $n=3^t-1$ for some $t >0$.
\end{theorem}
\begin{proof}
Assume that there are exactly $k$ codewords of the form
$3e_i+2e_j$ in $\T$. From Corollaries~\ref{cor:cover2ei}
and~\ref{cor:one32and4} and by Lemma~\ref{lem:disjoint32} the
lattice $\T$ contains a sublattice defined by these $k$ codewords
and $n-k$ codewords of the form $4 e_s$. The generator matrix
of this sublattice is a block-diagonal matrix with $k$~~$2 \times 2$
blocks of the form
$\begin{scriptsize} \left[ \begin{array}{cc} 3 & 2 \\ 0 & 4 \end{array} \right] \end{scriptsize}$
and $n-2k$~~$1 \times 1$ blocks of the
form~$\begin{scriptsize} \left[ \begin{array}{c} 4 \end{array} \right] \end{scriptsize}$.
The volume of
this sublattice is divided by the volume of the lattice $\T$.
The volume of the sublattice is $3^k 4^{n-k}$ and therefore, the volume of
the lattice $\T$ is of the form~$3^\ell 2^m$,
for some $\ell \geq 0$ and $m \geq 0$. On the otherhand the volume
of the lattice $\T$ is the volume of the shape $\Upsilon_n$, i.e. $2^n (n+1)$. By
Theorem~\ref{thm:odd_tiling} we have that if $n$ is odd then
$n=2^t-1$ for some $t > 0$. If $n$ is even then $n+1$ is odd and since
$3^\ell 2^m= 2^n (n+1)$ we must have that
$n=3^\ell-1$ for some $\ell >0$.
\end{proof}

%\begin{cor}
%If $\T'$ is a $\begin{tiny} \frac{1}{2}
%\end{tiny} \Z$-tiling with a $(0.5,n)$-cross obtained by a lattice then
%either $n=2^t-1$ or $n=3^t-1$ for some $t >0$.
%\end{cor}

\subsection{Tiling for even $n$}
\label{sec:Tiling_even}

In this subsection we will use the concept of packing triple system to
prove that if $n$ is even then $\T$ contains exactly $\frac{n}{2}$ codewords
of the form $3 e_r + 2e_s$, where the union of their nonzero coordinates is the
set of all $n$ coordinates. The structure of the codewords in $\T$ which was
proved in subsection~\ref{sec:Tiling_odd} and will be proved in this subsection,
combined with arguments based on reflections and translations
of the tiling, will imply a period 12 for the tiling when $n$ is even.
As a consequence we infer that if $n$ is even then
$n=3^t -1$, for some $t>0$.

A \emph{packing triple system} of order $n$ is a pair $(Q,\cB)$,
where $Q$ is an $n$-set and $\cB$ is a collection of 3-subsets of
$Q$, called \emph{blocks} such that each 2-subset of $Q$ is
contained in at most one block of $\cB$. Spencer~\cite{Spen}
proved that if $n \not\equiv 5~(\text{mod}~6)$ then
\begin{equation}
\label{eq:spencer} |\cB | \leq \left\lfloor \frac{n}{3}
\left\lfloor \frac{n-1}{2} \right\rfloor \right\rfloor ~.
\end{equation}

\begin{lemma}
\label{lem:cover_ij}
For each $1 \leq i < j \leq n$, the point $e_i + e_j$ is covered
by a codeword $X \in \T$, where $X=3e_i + 2e_j$ or $X=3e_j + 2e_i$
or $X \in \cD_3$, where $x_i=x_j=2$.
\end{lemma}
\begin{proof}
Follows from Lemmas~\ref{lem:cover_cond}
and~\ref{lem:codewords32} and the fract that for each nonzero codeword
$X \in \T$ we have $w_C(X) \geq 3$.
\end{proof}

Let
$$
\cF_1 \deff \{ \{ i,j \} ~:~ 3e_i +2e_j  \in \T \}
$$
and
$$
\cF_2 \deff \{ \{ i,j,k \} ~:~ 2e_i +2e_j +2e_k +\sum_{m \not\in
\{i,j,k \}} \alpha_m e_m \in \T, ~ \alpha_m \in \{0,1\} \} ~.
$$
Since $\T$ is a tiling it follows that each point $e_i + e_j$, $i
\neq j$, is covered by exactly one codeword of $\T$. As a
consequence of Lemma~\ref{lem:cover_ij},
we have that each
pair $\{ r,s \}$ is a subset of exactly one element from $\cF_1
\cup \cF_2$. Therefore, $\cF_2$ is a packing triple system of
order $n$.
\begin{theorem}
\label{thm:4mod6_tiling} If $\T$ is an integer tiling with
$\Upsilon_n$ then $n \not\equiv 4~(\text{mod}~6)$.
\end{theorem}
\begin{proof}
Assume $\T$ is an integer tiling with $\Upsilon_n$, $n \equiv
4~(\text{mod}~6)$. By (\ref{eq:spencer}) we have that
$$
|\cF_2| \leq \frac{n^2 -2n-2}{6}~.
$$
Since each pair $\{ i,j \} \subset \{ 1,2, \ldots ,n \}$ is
contained in either $\cF_1$ or $\cF_2$ it follows that
$$
|\cF_1| + 3 | \cF_2 | = \binom{n}{2}~.
$$
Hence, $|\cF_1| \geq \frac{n}{2} +1$.
Lemma~\ref{lem:disjoint32} implies that $|\cF_1| \leq \frac{n}{2}$, a contradiction.
\end{proof}
By using the same arguments as in the proof of
Theorem~\ref{thm:4mod6_tiling} we have that if $n \equiv 0$ or
$2~(\text{mod}~6)$ then $|\cF_1| \geq \frac{n}{2}$. Hence, by
Lemma~\ref{lem:disjoint32} we infer
\begin{lemma}
\label{lem:half_n_32} If $n$ is even and $\T$ is an integer tiling
with $\Upsilon_n$, then there are exactly $\frac{n}{2}$ codewords
of the form $3e_r +2e_s$.
\end{lemma}
Combing Lemmas~\ref{lem:disjoint32} and~\ref{lem:half_n_32} we
infer
\begin{cor}
\label{cor:32cover_n} If $n$ is even and $\T$ is an integer tiling
with $\Upsilon_n$, then there are exactly $\frac{n}{2}$ codewords of the
form $3e_r+2e_s$ and the set $\{ i ~:~ 3e_i +2e_j \in \T$ or
$3e_j+2e_i \in \T \}$ contains all the integers between 1 and $n$.
\end{cor}

Let $\T'$ be the tiling of $\Z^n$ with $\Upsilon_n$ defined by
$\T' \deff \{ X ~:~ -X \in \T \}$. Since $\T'$ is a tiling of
$\Z^n$ with $\Upsilon_n$, it follows that the lemmas and the
corollaries of Section~\ref{sec:nonexistence} hold also for
$\T'$. They imply new lemmas and corollaries for $\T$. For
example we have
\begin{cor}
For each $r$, $1 \leq r \leq n$, the point $-2 e_r$ is covered by
a codeword $X \in \T$, where either $X=-4e_r$ or $X=-3e_r -2e_s$
for some $s \neq r$.
\end{cor}
In a similar way we can define $2^n$ tilings of $\Z^n$
with $\Upsilon_n$. For $A=(a_1 , a_2 , \ldots , a_n )$, where $a_i
\in \{ -1 ,1 \}$, let $\T_A$ be the tiling of $\Z^n$ with
$\Upsilon_n$ defined by $\T_A \deff \{ (x_1,x_2, \ldots ,x_n) ~:$
$(a_1 x_1 , a_2 x_2 , \ldots , a_n x_n ) \in \T \}$. As for
$\T'=\T_{(-1,-1,\ldots,-1)}$, each lemma and each corollary of
Section~\ref{sec:nonexistence} holds for $\T_A$ and thus implies
new claims on $\T$. Without loss of generality we assume (based on
Lemma~\ref{lem:permute}, Corollaries~\ref{cor:one32and4}
and~\ref{cor:32cover_n}) that $3e_{2i-1} +2e_{2i} \in \T$ and
$4e_{2i} \in \T$, for all $1 \leq i \leq \frac{n}{2}$.

\begin{lemma}
\label{lem:23implies-4} If $X=3e_r +2e_s \in \T$ then $-4e_s \in
\T$.
\end{lemma}
\begin{proof}
Without loss of generality we will prove the claim for $r=1$ and $s=2$; let
$A=(1,-1,1,\ldots,1)$. Since $3e_{2i-1} +2e_{2i} \in \T$, for all
$2 \leq i \leq \frac{n}{2}$, it follows that $3e_{2i-1} +2e_{2i}
\in \T_A$, for all $2 \leq i \leq \frac{n}{2}$, and by
Corollary~\ref{cor:32cover_n} we have that either $3e_1 + 2e_2 \in
\T_A$ or $2e_1 + 3e_2 \in \T_A$. If $2e_1+3e_2 \in \T_A$ then
Corollary~\ref{cor:one32and4} implies that $Y=4e_1 \in \T_A$.
Therefore, $Y=4e_1 \in \T$, and since $d_C(X,Y)=1$ we have a
contradiction. Hence, $3e_1 + 2e_2 \in \T_A$, and therefore, by
Corollary~\ref{cor:one32and4} we have that $4e_2 \in \T_A$, i.e.
$-4e_2 \in \T$.
\end{proof}
\begin{cor}
\label{cor:-4iff4} $4e_s \in \T$ if and only if $-4e_s \in \T$.
\end{cor}
\begin{lemma}
\label{lem:23implies-23} If $X=3e_r +2e_s \in \T$ then $-3e_r
-2e_s \in \T$.
\end{lemma}
\begin{proof}
Without loss of generality we will prove the claim for $r=1$ and $s=2$; let
$A=(-1,-1,1,\ldots,1)$. Since $3e_{2i-1} +2e_{2i} \in \T$, for all
$2 \leq i \leq \frac{n}{2}$, it follows that $3e_{2i-1} +2e_{2i}
\in \T_A$, for all $2 \leq i \leq \frac{n}{2}$, and by
Corollary~\ref{cor:32cover_n} we have that either $3e_1 + 2e_2 \in
\T_A$ or $2e_1 + 3e_2 \in \T_A$. If $2e_1+3e_2 \in \T_A$ then
Lemma~\ref{lem:23implies-4} implies that $-4e_1 \in \T_A$.
Therefore, $Y=4e_1 \in \T$, and since $d_C(X,Y)=1$ we have a
contradiction. Hence, $3e_1 + 2e_2 \in \T_A$, and therefore we
have that $-3e_1 - 2e_2 \in \T$.
\end{proof}
\begin{cor}
\label{cor:23iff-23} $3e_r +2e_s \in \T$ if and only if $-3e_r
-2e_s \in \T$.
\end{cor}

\begin{lemma}
\label{lem:cw12e} If $3e_r +2e_s \in \T$ then $12e_r,12e_s \in
\T$.
\end{lemma}
\begin{proof}
By Corollary~\ref{cor:one32and4} we have that $4e_s \in \T$.
The translation $\T_1
= -4e_s +\T$ is a tiling with $\Upsilon_n$ for which ${\bf 0},
-4e_s \in \T_1$. It follows by Corollary~\ref{cor:-4iff4} that
$4e_s \in \T_1$ and hence $8e_s \in \T$. Similarly, $12e_s \in
\T$.

Similarly, by Corollary~\ref{cor:23iff-23} we have that ${\bf 0},
3e_r +2e_s \in \T$ implies that $6e_r +4e_s , 9e_r+6e_s,12e_r+8e_s
\in \T$. The translation $\T_1 = -12e_r-8e_s +\T$ is a tiling with $\Upsilon_n$
for which ${\bf 0},-3e_r-2e_s \in \T_1$. By
Corollary~\ref{cor:23iff-23} and Lemma~\ref{lem:23implies-4} we
have that $-4e_s \in \T_1$, and hence $12e_r +4e_s \in \T$.
Similarly, by Corollary~\ref{cor:-4iff4} we have $12e_r +4e_s ,
12e_r +8e_s \in \T$ which implies that $12e_r \in \T$.
\end{proof}

\begin{cor}
\label{cor:period12} If $n$ is even and $\T$ is an integer tiling
with $\Upsilon_n$, then $\T$ is a periodic tiling with period 12.
\end{cor}
\begin{theorem}
\label{thm:even_tiling} If $\T$ is an integer tiling with
$\Upsilon_n$, where $n$ is an even integer, then $n=3^t-1$ for
some $t > 0$.
\end{theorem}
\begin{proof}
By Corollary~\ref{cor:period12} we have that $\T$ is a periodic
tiling with period 12. Therefore, the size of $\Upsilon_n$ divides
$12^n$. The size of $\Upsilon_n$ is $2^n (n+1)$ and hence $n+1$ divides $2^n
3^n$. Since $n$ is even it follows that $n+1$ is odd and thus
$n=3^t-1$ for some $t> 0$.
\end{proof}

Theorems~\ref{thm:odd_tiling} and~\ref{thm:even_tiling} are
combined to obtain

\begin{cor}
If $\T$ is an integer tiling with $\Upsilon_n$, then either
$n=2^t-1$ or $n=3^t-1$, for some $t >0$.
\end{cor}
\begin{cor}
If $\T$ is a $\begin{tiny} \frac{1}{2}
\end{tiny} \Z$-tiling with a $(0.5,n)$-cross, then either
$n=2^t-1$ or $n=3^t-1$, for some $t >0$.
\end{cor}

%@@@@@@@@@@@@@@@@@@@@@@@@@@@@@@@@@@@@@@@@@@@@@@@@@@@@@@@@@@@@@@@@@@@@@@@@%
%                                                                        %
%   6. Connections to perfect codes                                      %
%                                                                        %
%@@@@@@@@@@@@@@@@@@@@@@@@@@@@@@@@@@@@@@@@@@@@@@@@@@@@@@@@@@@@@@@@@@@@@@@@%
\section{Tilings based on Perfect Codes}
\vspace{-.25ex} \label{sec:perfect}

In Section~\ref{sec:nonexistence} we proved that a $\begin{tiny}
\frac{1}{2} \end{tiny} \Z$-tiling with $(0.5,n)$-cross exists only
if $n=2^t-1$ or $n=3^t-1$, for some $t >0$. In this section we will prove
that this necessary condition is also sufficient. Surprisingly,
two constructions which produce the related tilings are based on
perfect codes in the Hamming scheme. If $n=2^t-1$ then the perfect code is
binary of length~$n$ and the construction of the tiling is very
simple. If $n=3^t-1$ then the perfect code is ternary of length~$\frac{n}{2}$.

We will refer only to perfect codes with minimum Hamming distance
three. A code $\cC$ has \emph{minimum Hamming distance} $d$ if for every
two distinct codewords $X,Y \in \cC$ we have $d_H(X,Y) \geq d$.
The minimum Hamming distance of $\cC$ will be denoted by
$d_H(\cC)$. Similarly, we define the \emph{minimum cross distance} of a
code. A code $\cC$ of length $n$ over $\Z_q$, with minimum Hamming
distance 3, is called \emph{perfect} if for each word $A \in
\Z_q^n$ there exists a codeword $X \in \cC$ such that $d_H (A,X)
\leq 1$. Such a code is also called a single-error-correcting
perfect code since it is capable to correct a single transmission
error~\cite{McSl77}. The \emph{sphere} of radius $\rho$ centered
at $A=(a_1,a_2,\ldots,a_n)$ is the set $\{ B \in \Z_q^n ~:~ d_H
(A,B) \leq \rho \}$. The code $\cC$ is a single-error-correcting perfect
code if and only if $\cC$ is a tiling of $\Z_q^n$ with a sphere of
radius one. Binary ($q=2$) perfect codes exists if and only if
$n=2^t-1$, where $t > 0$. Ternary ($q=3$) perfect codes exists if and
only if $n=\frac{3^t-1}{2}$, where $t > 0$. These are the only perfect
codes which are of interest in this section. Finally, we note that
a perfect code is identified by its size, its minimum distance,
and the fact that each element of $\Z_q^n$ is covered by at least
one codeword. One can easily verify that given any two of these
parameters one can determine whether the code is perfect or not perfect. This
fact will be used throughout this section.

\begin{remark}
\label{rem:perfect} A perfect code $\cC$ of length $n$ over $\Z_q$
is known to exist if $q$ is a power of a prime and $n = \frac{q^t
-1}{q-1}$, where $t
>0$. The related sphere of radius one can be viewed as a
$(q-1,n)$-semicross or as a $(\frac{q-1}{2},n)$-cross. Thus, these
perfect codes form tilings with the related semicrosses and
crosses. Only if $q$ is a prime some of the known tilings are
lattice tilings (they are related to linear perfect codes). If $q$
is not a prime then the tiling of $\Z^n$ is done first by using
any one-to-one mapping between GF($q$) (on which the codes are
defined) and $\Z_q$. Tilings of this type have applications in
flash memories~\cite{Sch11}. If $q=2$ then $\cC$ is a tiling of
$\Z_2^n$ with $(0.5,n)$-cross and $E(\cC)$ forms a tiling of
$\Z^n$ with $(0.5,n)$-cross.
\end{remark}

\subsection{Binary Perfect Codes}

Since the size of of a sphere with radius one in $\Z_2^n$ is $n+1$,
it follows that a binary perfect code of length $n=2^t-1$ has $2^{n-t}$
codewords.

\begin{theorem}
\label{thm:binary_cross} There exists an one-to-one correspondence
between the set of binary perfect codes of length $n=2^t -1$ and
the set of integer tilings with $\Upsilon_n$ in which each
codeword has only even entries.
\end{theorem}
\begin{proof}
Note first, that by Corollary~\ref{cor:odd4} a tiling $\T$ of
$\Z^n$ with $\Upsilon_n$ is periodic with period 4 and hence it
can be reduced to a tiling of $\Z_4^n$ with $\Upsilon_n$.

The size of an $(1,n)$-semicross is equal the size of a
$(0.5,n)$-cross. It implies that the number of codewords in a
binary single-error-correcting perfect code $\cC$ of length $n=2^t
-1$ is equal the number of codewords in a tiling $\T$ of $\Z_4^n$
with $\Upsilon_n$. If $X,Y \in \{ 0,2 \}^n$ then $0.5X$ and $0.5Y$
are binary words and it is easy to verify that $d_C(X,Y) = d_H
(0.5X,0.5Y)$.

Therefore, if $\cC$ is a binary perfect code of length $n=2^t -1$
then $2 E(\cC)$ is a tiling of $\Z^n$ with $\Upsilon_n$ in which
each codeword has only even entries. Similarly, if $\T$ is a
tiling of $\Z^n$ with $\Upsilon_n$, in which each codeword has
only even entries, then $0.5 \T \cap \{ 0,1 \}^n$ is a binary
perfect code.
\end{proof}

\begin{cor}
There exists an one-to-one correspondence between the set of
binary perfect codes of length $n=2^t -1$ and the set of integer
tilings with $(0.5,n)$-cross.
\end{cor}

Do there exists any integer tilings with $\Upsilon_n$, where $n=2^t-1$,
except for those implied by Theorem~\ref{thm:binary_cross}? The
answer is that there exist many such tilings. Let $\cC$ be a
binary code of length $n$. Its \emph{punctured} code $\cC'$ of
length $n-1$ is defined by $\cC' \deff \{ c ~:~ (c,x) \in \cC, ~ x
\in \{0,1\} \}$.

\begin{construction}
Let $\cC$ be a binary perfect code of length $n$ and $\cC'$ its
punctured code. Let $\cC'_e$ and $\cC'_o$ be the set of codewords
from $\cC'$ with even weight and odd weight, respectively. We
define a code $\cC^*
\deff \cC_1^* \cup \cC_2^*$ over $\Z_4^n$, where
$$
\cC_1^* \deff \{ (2c,2x) ~:~ c \in \cC'_e ,~ (c,x) \in \cC \} ~
\mbox{and} ~ \cC_2^* \deff \{ (2c,2x+1) ~:~ c \in \cC'_o ,~ (c,x)
\in \cC \}~.
$$
\end{construction}
\begin{theorem}
$E(\cC^*)$ defines a tiling of $\Z^n$ with $\Upsilon_n$, in which
not all entries are even.
\end{theorem}
\begin{proof}
Since $d_H(\cC) =3$ it follows that $d_H(\cC') =d_H(\cC'_e)
=d_H(\cC'_o) =2$ and $d_C (\cC_1^*)=d_C(\cC_2^*)=3$. If $\tilde{c}_1
\in \cC'_e$ and $\tilde{c}_2 \in \cC'_o$ then $d_H (\tilde{c}_1 ,
\tilde{c}_2)$ is an odd integer. Hence, since $d_H(\cC')=2$, it
follows that $d_H (\tilde{c}_1 , \tilde{c}_2) \geq 3$. Therefore,
if $\tilde{c}_1^* \in \cC_1$ and $\tilde{c}_2^* \in \cC_2$ then $d_C
(\tilde{c}_1^* , \tilde{c}_2^*) \geq 3$ and thus $d_C (\cC^*) \geq
3$. The minimum distance of the code $\cC^*$ and its number of
codewords implies that $\cC^*$ is a tiling of $\Z_4^n$ with
$\Upsilon_n$. It is easy to verify that $\cC'_o$ has at least one
codeword (in fact it can be proved that it contains exactly half
of the codewords) and hence the last entry in at least one of the
codewords of $\cC^*$ is 1 or 3.
\end{proof}

\begin{example}
The following code forms a tiling of $\Z_4^7$ with $\Upsilon_7$:
$$\begin{array}{cccc}
0000000 & 0000222 & 2222000 & 2222222 \\
2200201 & 2200023 & 0022201 & 0022023 \\
2020021 & 2020203 & 0202021 & 0202203 \\
2002002 & 2002220 & 0220002 & 0220220
\end{array}$$
\end{example}

\begin{remark}
Let $\xi$ be a mapping from $\Z_4$ to $\Z_2$ defined by
$\xi(0)=\xi(1)=0$, $\xi(2)=\xi(3)=1$. If $\T$ forms a tiling of
$\Z_4^n$ with $\Upsilon_n$ then the code $\cC = \{ \xi (X) ~:~ X
\in \T \}$, where $\xi (x_1,x_2,\ldots,x_n) =
(\xi(x_1),\xi(x_2),\ldots,\xi(x_n))$ is a binary perfect code of
length $n$.
\end{remark}
\begin{remark}
By Corollary~\ref{cor:odd4} an integer tiling with $\Upsilon_n$,
where $n$ is odd, has period 4. Hence, the related $\begin{tiny}
\frac{1}{2} \end{tiny} \Z$-tiling $\T$ with $(0.5,n)$-cross has
period 2. It implies that this tiling is also a tiling with the
$(1,n)$-semicross (even if $\T$ is not a $\Z$-tiling).
\end{remark}

\subsection{Ternary Perfect Codes}

Let $n=3^t-1$, where $t > 0$, and $\nu = \frac{n}{2}$. Since the
size of a sphere with radius one in $\Z_3^\nu$ is $2\nu+1$,
it follows that a ternary perfect code of length $\nu$ has $3^{\nu-t}$
codewords. Let ${\Lambda}_n$
be the lattice generated by the basis $\{ 3e_{2i-1} +2e_{2i} ~:~ 1
\leq i \leq \nu \} \cup \{ 4e_{2i} ~:~ 1 \leq i \leq \nu \}$. Let
$G_n$ be the quotient group ${\mathbb{Z}}^n / {\Lambda}_n$. Recall that
we denote the set of integers in $\Z_p$ without the group structure
by $\tilde{\Z}_p \deff \{ 0,1, \ldots , p-1 \}$.
The following lemma can be readily verified.

\begin{lemma}
\label{lem:rep_2} The group $G_2$ has size 12 and the 12 representatives of
elements from $G_2$ (the cosets of $\Lambda_2$ in $\Z^2$) can be
taken as $\tilde{\Z}_3 \times \tilde{\Z}_4$.
\end{lemma}
Let $(\tilde{\Z}_3 \times \tilde{\Z}_4)^m \deff
\underbrace{(\tilde{\Z}_3 \times \tilde{\Z}_4) \times
(\tilde{\Z}_3 \times \tilde{\Z}_4) \times ~ \cdots ~ \times
(\tilde{\Z}_3 \times \tilde{\Z}_4)}_{m\textrm{ times}}$.
\begin{cor}
\label{cor:rep_n} The group $G_n$ has size $12^\nu$ and the $12^\nu$
representatives of elements from $G_n$ (the cosets of $\Lambda_n$
in $\Z^n$) can be taken as the elements of
$(\tilde{\Z}_3 \times \tilde{\Z}_4)^\nu$.
\end{cor}

Consider the mapping $\Phi:{\mathbb{Z}}_3^{\nu}\rightarrow G_n$
defined by
$$
\Phi(x_1,x_2,...,x_{\nu})=(\phi(x_1),\phi(x_2),...,\phi(x_{\nu}))~,
$$
where $\phi:{\mathbb{Z}}_3\rightarrow G_2$ is a mapping defined by
\begin{equation*}
\phi(x)=\begin{cases}
                   (0,0) & ~~\mbox{if}~x=0\\
                    (1,2) & ~~\mbox{if}~x=1\\
                    (2,0) & ~~\mbox{if}~x=2
                 \end{cases}~.
\end{equation*}\\
It is easy to verify that both $\phi$ and $\Phi$ are injective group
homomorphisms.

Let $\cC$ be a ternary perfect code of length $\nu$ with
$3^{\nu-t}$ codewords, and let $\Phi (\cC) \deff \{ \Phi
(\tilde{c}) : \tilde{c} \in \cC \}$. Since the elements of $\Phi
(\cC)$ are representatives of elements of $G_n$ (see
Corollary~\ref{cor:rep_n}) it follows that the elements of $\Phi
(\cC)$ can be considered as elements in $\Z^n$. Let $\T_n \deff
\Phi (\cC) + \Lambda_n$.

\begin{theorem}
\label{thm:tiling3n} The set $\T_n$ is a tiling of $\Z^n$ with
$\Upsilon_n$.
\end{theorem}
\begin{proof}
Clearly, $\Lambda_n$ is a lattice with period 12 and hence $\T_n$
is a periodic code of $\Z^n$ with period 12. Therefore, without loss of generality
we can restrict our discussion to $\Z_{12}^n$. i.e. codewords of
$\T_n \cap \tilde{\Z}_{12}^n$. Since $|\Upsilon_n| = 2^{2\nu} 3^t$
it follows that the size of the tiling $\T_n$ in
$\tilde{\Z}_{12}^n$, $| \T_n \cap \tilde{\Z}_{12}^n |$, should be
$2^{2\nu} 3^{2\nu-t}$. To prove that $\T_n$ is a tiling of $\Z^n$
with $\Upsilon_n$ we will show that the size of $\T_n \cap
\tilde{\Z}_{12}^n$ is $2^{2\nu} 3^{2\nu-t}$ and we will prove that
each point of $\Z^n$ is covered by an element of $\T_n$.

\noindent {\bf Claim:} For any two codewords $\tilde{c}_1 ,
\tilde{c}_2 \in \cC$, and two lattice points $Y_1 ,Y_2 \in
\Lambda_n$, we have $\Phi (\tilde{c}_1)+ Y_1 \neq \Phi
(\tilde{c}_2)+ Y_2$, unless $\tilde{c}_1 = \tilde{c}_2$ and $Y_1 =
Y_2$.

\noindent {\bf Proof:} Assume that $\Phi (\tilde{c}_1)+ Y_1 = \Phi
(\tilde{c}_2)+ Y_2$, i.e. $\Phi (\tilde{c}_1) - \Phi (\tilde{c}_2)
= Y_2 -Y_1$, $\tilde{c}_1 , \tilde{c}_2 \in \cC$ and $Y_1 ,Y_2 \in
\Lambda_n$. Hence, $Y_2 - Y_1 =(\alpha_1 , \ldots , \alpha_n)$ is a
lattice point and unless $Y_1 = Y_2$ we have that for at least one
$i$, $| \alpha_i | > 2$. Denote $\Phi (\tilde{c}_1) - \Phi (\tilde{c}_2)
= (\beta_1 , \ldots , \beta_n)$. By the definition of $\Phi$, for
each $i$, $1 \leq i \leq n$, we have $| \beta_i | \leq 2$.
Therefore, $Y_1 = Y_2$ and $\Phi (\tilde{c}_1) = \Phi
(\tilde{c}_2)$ and since $\Phi$ is an injective mapping it implies
that $\tilde{c}_1 = \tilde{c}_2$ and the claim is proved.

The claim implies that $| \T_n \cap \tilde{\Z}_{12}^n | = |\Phi
(\cC)| \cdot |\Lambda_n \cap \tilde{\Z}_{12}^n |$. Since $\Phi$ is
an injective mapping we also have that $|\Phi (\cC)| = |\cC|$.
Since $\Lambda_n$ has period 12 and $V(\Lambda_n) = 12^\nu$ it
follows that $|\Lambda_n \cap \tilde{\Z}_{12}^n | = 12^\nu$.
Therefore,
$$
| \T_n \cap \tilde{\Z}_{12}^n | = |\Phi (\cC)| \cdot |\Lambda_n
\cap \tilde{\Z}_{12}^n | = |\cC| \cdot |\Lambda_n \cap
\tilde{\Z}_{12}^n | = 3^{\nu-t} 12^\nu =2^{2\nu} 3^{2\nu-t}
$$
as required.

To show that every point of $\Z^n$ is covered by an element of
$\T_n$ we first partition the elements of $\tilde{\Z}_3 \times
\tilde{\Z}_4$ into three classes:
$$
\begin{array}{c}
[(0,0)]=\{(0,0),(0,3),(2,2),(2,1)\} \\
{[(1,2)]}=\{(1,2),(1,1),(0,1),(0,2)\} \\
{[(2,0)]}= \{(2,0),(1,3),(2,3),(1,0)\}
\end{array} ~,
$$
The following two properties are readily verified (as can be
verified from Table~\ref{tab:t1}).
\begin{itemize}
\item[{\bf (P.1)}] For each element $(x_1,x_2)$ in a class
$[(y_1,y_2)]$ there exists an element $(u_1,u_2) \in \Lambda_2$
such that $u_i+y_i\in\{x_i,x_i+1\}$, for $i\in\{1,2\}$.

\item[{\bf (P.2)}] For each element $(x_1,x_2) \in \tilde{\Z}_3
\times \tilde{\Z}_4$ and each class $[(y_1,y_2)]$ there exists an
element $(u_1,u_2) \in \Lambda_2$ such that $u_i+y_i\in
\{x_i-1,x_i,x_i +1,x_i+2\}$, for $i\in\{1,2\}$, and for at most one
$i$ we have $u_i+y_i\in\{x_i-1,x_i+2\}$.
\end{itemize}
\begin{table}
\centering
\begin{tabular}{|c|c|c|c|c|c|c|}
\hline
& $class$~[(0,0)]   & (0,0) & (0,3) & (2,2) & (2,1) & $\longleftarrow$ $(x_1,x_2)$ \\
\hline
{\bf (P.1)} & $[(y_1,y_2)]=[(0,0)]$ &(0,0) & (0,4) & (3,2) & (3,2) & $\longleftarrow$ $(u_1,u_2) + (y_1,y_2)$ \\
\hline
{\bf (P.2)} & $[(y_1,y_2)]=[(1,2)]$& (1,2) & (1,2) & (1,2) & (1,2) & $\longleftarrow$ $(u_1,u_2) + (y_1,y_2)$ \\
\hline
{\bf (P.2)} & $[(y_1,y_2)]=[(2,0)]$ & (2,0) & (2,4) & (2,4) &(2,0) & $\longleftarrow$ $(u_1,u_2)+ (y_1,y_2)$ \\
\hline \hline
&  $class$~[(1,2)] &(1,2)& (1,1)&(0,1) & (0,2) & $\longleftarrow$ $(x_1,x_2)$\\
\hline
{\bf (P.2)} &   $[(y_1,y_2)]=[(0,0)]$& (3,2) &(3,2) & (0,0) & (0,4) & $\longleftarrow$ $(u_1,u_2)+ (y_1,y_2)$ \\
\hline
{\bf (P.1)} &   $[(y_1,y_2)]=[(1,2)]$& (1,2) & (1,2) & (1,2) & (1,2) & $\longleftarrow$ $(u_1,u_2)+ (y_1,y_2)$ \\
\hline
{\bf (P.2)} &   $[(y_1,y_2)]=[(2,0)]$& (2,4) & (2,0) & (-1,2) & (-1,2) & $\longleftarrow$ $(u_1,u_2)+ (y_1,y_2)$ \\
\hline \hline
&   $class$~[(2,0)] & (2,0) & (1,3) & (2,3) &(1,0) & $\longleftarrow$ $(x_1,x_2)$\\
\hline
{\bf (P.2)} &   $[(y_1,y_2)]=[(0,0)]$ & (3,2) & (0,4) & (3,2) &(0,0) & $\longleftarrow$ $(u_1,u_2)+ (y_1,y_2)$ \\
\hline
{\bf (P.2)} & $[(y_1,y_2)]=[(1,2)]$ & (4,0) & (1,2) & (4,4) & (1,2) & $\longleftarrow$ $(u_1,u_2)+ (y_1,y_2)$ \\
\hline
{\bf (P.1)} & $[(y_1,y_2)]=[(2,0)]$ & (2,0) & (2,4) & (2,4) & (2,0) & $\longleftarrow$ $(u_1,u_2)+ (y_1,y_2)$ \\
\hline
\end{tabular}
 \caption{}
 \label{tab:t1}
\end{table}

%$$
%\begin{array}{|c|c|c|c|c|c|c|}
%\hline
%& $class$~[(0,0)]   & (0,0) & (0,3) & (2,2) & (2,1) & \longleftarrow (x_1,x_2) \\
%\hline
%{\bf (P.1)} & [(y_1,y_2)]=[(0,0)] &(0,0) & (0,4) & (3,2) & (3,2) & \longleftarrow (u_1,u_2)+ (y_1,y_2) \\
%\hline
%{\bf (P.2)} & [(y_1,y_2)]=[(1,2)]& (1,2) & (1,2) & (1,2) & (1,2) & \longleftarrow (u_1,u_2)+ (y_1,y_2) \\
%\hline
%{\bf (P.2)} & [(y_1,y_2)]=[(2,0)] & (2,0) & (2,4) & (2,4) &(2,0) & \longleftarrow (u_1,u_2)+ (y_1,y_2) \\
%\hline \hline
%&  $class$~[(1,2)] &(1,2)& (1,1)&(0,1) & (0,2) & \longleftarrow (x_1,x_2)\\
%\hline
%{\bf (P.2)} &   [(y_1,y_2)]=[(0,0)]& (3,2) &(3,2) & (0,0) & (0,4) & \longleftarrow (u_1,u_2)+ (y_1,y_2) \\
%\hline
%{\bf (P.1)} &   [(y_1,y_2)]=[(1,2)]& (1,2) & (1,2) & (1,2) & (1,2) & \longleftarrow (u_1,u_2)+ (y_1,y_2) \\
%\hline
%{\bf (P.2)} &   [(y_1,y_2)]=[(2,0)]& (2,4) & (2,0) & (-1,2) & (-1,2) & \longleftarrow (u_1,u_2)+ (y_1,y_2) \\
%\hline \hline
%&   $class$~[(2,0)] & (2,0) & (1,3) & (2,3) &(1,0) & \longleftarrow (x_1,x_2)\\
%\hline
%{\bf (P.2)} &   [(y_1,y_2)]=[(0,0)] & (3,2) & (0,4) & (3,2) &(0,0) & \longleftarrow (u_1,u_2)+ (y_1,y_2) \\
%\hline
%{\bf (P.2)} & [(y_1,y_2)]=[(1,2)] & (4,0) & (1,2) & (4,4) & (1,2) & \longleftarrow (u_1,u_2)+ (y_1,y_2) \\
%\hline
%{\bf (P.1)} & [(y_1,y_2)]=[(2,0)] & (2,0) & (2,4) & (2,4) & (2,0) & \longleftarrow (u_1,u_2)+ (y_1,y_2) \\
%\hline
%\end{array}
%$$

%\end{center}
Consider the mapping $\Psi: (\tilde{\Z}_3 \times \tilde{\Z}_4)^\nu
\rightarrow {\mathbb{Z}}_3^{\nu}$ defined by
$$
\Psi(x_1,x_2,...,x_n)=(\psi(x_1,x_2),\psi(x_3,x_4),...,\psi(x_{n-1},x_n))~,
$$
where $\psi:{\tilde{\Z}_3 \times \tilde{\Z}_4 \rightarrow
\mathbb{Z}}_3$ is a mapping defined by
\begin{equation*}
\psi(x_1,x_2)=\begin{cases}
                   0 & ~~\mbox{if}~(x_1,x_2) \in [(0,0)]\\
                    1 & ~~\mbox{if}~(x_1,x_2) \in [(1,2)]\\
                    2 & ~~\mbox{if}~(x_1,x_2) \in [(2,0)]
                 \end{cases}~.
\end{equation*}
For a given point $A = (a_1 , a_2 , \ldots , a_n ) \in \Z^n$ we
will exhibit a point $X \in \T_n$ which covers $A$.  By
Corollary~\ref{cor:rep_n} we have that there exists an element $Y
\in \Lambda_n$ such that $A+Y \in (\tilde{\Z}_3 \times
\tilde{\Z}_4)^\nu$. Let $B=A+Y=(b_1,b_2, \ldots , b_n)$ and let
$\Psi (B) =(\alpha_1 , \alpha_2 , \ldots , \alpha_{\nu} ) \in
\Z_3^\nu$. Since $\cC$ is a perfect code of length $\nu$ over
$\Z_3$ it follows that there exists a codeword $(c_1 , c_2 ,
\ldots , c_\nu ) \in \cC$ such that $d_H ((\alpha_1 , \alpha_2 ,
\ldots , \alpha_{\nu} ),(c_1 , c_2 , \ldots , c_\nu )) \leq 1$.
Let $(\gamma_1 , \gamma_2 , \ldots , \gamma_n ) = \Phi (c_1 , c_2
, \ldots , c_\nu )$. Note that by the definitions of $\Phi$ and
$\Psi$ it follows that $(b_{2i-1},b_{2i})$ and $\phi(\alpha_i)$
are in the same class, for all $1\leq i\leq \nu$. Now, we
distinguish between two cases:

\noindent {\bf Case 1:} If $(\alpha_1 , \alpha_2 , \ldots ,
\alpha_{\nu} )=(c_1 , c_2 , \ldots , c_\nu )$ then by property {\bf (P.1)}
there exists an element $(u_1,u_2, \ldots u_n) \in \Lambda_n$ such that
$u_i+\gamma_i\in \{b_i,b_i+1\}$, for $1\leq i\leq n$. Therefore, by
Lemma \ref{lem:cover_cond} we have that
$(u_1,u_2,\ldots,u_n)+(\gamma_1 , \gamma_2 , \ldots , \gamma_n )$
covers $B$ and hence the required $X$ is $(u_1,u_2, \ldots
u_n)+(\gamma_1 , \gamma_2 , \ldots , \gamma_n )-Y$.

\noindent {\bf Case 2:} If $(\alpha_1 , \alpha_2 , \ldots ,
\alpha_{\nu} ) \neq (c_1 , c_2 , \ldots , c_\nu )$ then
$d_H((\alpha_1 , \alpha_2 , \ldots , \alpha_{\nu} ),(c_1 , c_2 ,
\ldots , c_\nu )) =1$ and hence there exists exactly one coordinate
$s$ such that $\alpha_s \neq c_s$. By properties {\bf (P.1)} and {\bf (P.2)}
there exists an element $(u_1,u_2, \ldots , u_n) \in \Lambda_n$ such that
$u_i+\gamma_i\in \{b_i-1,b_i,b_i+1,b_i+2\}$, for $1\leq i\leq n$, and
for at most one $i$ we have $u_i+\gamma_i\in \{b_i-1,b_i+2\}$.
Therefore, by Lemma~\ref{lem:cover_cond} we have that
$(u_1,u_2,\ldots,u_n)+(\gamma_1 , \gamma_2 , \ldots , \gamma_n )$
covers $B$ and hence the required $X$ is $(u_1,u_2, \ldots ,
u_n)+(\gamma_1 , \gamma_2 , \ldots , \gamma_n )-Y$.

Since we proved that the size of $\T_n \cap
\tilde{\Z}_{12}^n$ is $2^{2\nu} 3^{2\nu-t}$ and
each point of $\Z^n$ is covered by an element of $\T_n$,
the theorem is proved.
\end{proof}

\begin{theorem}
If $\cC$ is a linear code then $\T_n$ is a lattice tiling.
\end{theorem}
\begin{proof}
Follows immediately from Theorem~\ref{thm:tiling3n} and the facts
that $\cC$ is a linear code and $\Phi$ is a group homomorphism.
\end{proof}

%@@@@@@@@@@@@@@@@@@@@@@@@@@@@@@@@@@@@@@@@@@@@@@@@@@@@@@@@@@@@@@@@@@@@@@@@%
%                                                                        %
%   7. Conclusion and Open Problems                                      %
%                                                                        %
%@@@@@@@@@@@@@@@@@@@@@@@@@@@@@@@@@@@@@@@@@@@@@@@@@@@@@@@@@@@@@@@@@@@@@@@@%

\begin{center}
{\bf Acknowledgement}
\end{center}
We would like to thank Italo J. Dejter for bringing the connection
to the perfect dominating set to our attention. We are grateful
to an anonymous reviewer for many important remarks which have
improved considerably the presentation of the paper.

%
%*************************** References *********************************
%

\end{document}